\documentclass[10pt,reqno,oneside]{amsproc}
\title[Inflow-outflow via analyticity]{The inviscid inflow-outflow problem via analyticity}
\author[I.~Kukavica]{Igor Kukavica}
\address{Department of Mathematics, University of Southern California, Los Angeles, CA 90089}
\email{kukavica@usc.edu}
\author[W.~S.~O\.za\'nski]{Wojciech O\.za\'nski}
\address{Department of Mathematics, Florida State University, Tallahassee, FL 32301}
\email{wozanski@fsu.edu}
\author[M.~Sammartino]{Marco Sammartino}
\address{Department of Mathematics, Via Archirafi, 34 Palermo Italy}
\email{marco.sammartino@unipa.it}

\chardef\forshowkeys=0
  \chardef\refcheck=0
  \chardef\showllabel=0
  \chardef\sketches=0
\usepackage{enumitem}
\usepackage{datetime}
\usepackage{fancyhdr}
\usepackage{comment}
\allowdisplaybreaks
\ifnum\forshowkeys=1

  \usepackage[notref,notcite,color]{showkeys}
\fi
\usepackage[margin=1in]{geometry}
\usepackage{amsmath, amsthm, amssymb,mathtools}
\usepackage{times}
\usepackage{graphicx}
\usepackage[usenames,dvipsnames,svgnames,table]{xcolor}
\usepackage{marginnote}
\usepackage[unicode,breaklinks=true,colorlinks=true,linkcolor=black,urlcolor=blue,citecolor=black]{hyperref}
\usepackage[most]{tcolorbox}
\usepackage{tikz}

\ifnum\refcheck=1
  \usepackage{refcheck}
\fi
\begin{document}
\def\YY{X}
\def\na{\nabla }
\def\OO{\mathcal O}
\def\SS{\mathbb S}
\def\CC{\mathbb C}
\def\RR{\mathbb R}
\def\TT{\mathbb T}
\def\ZZ{\mathbb Z}
\def\HH{\mathbb H}
\def\RSZ{\mathcal R}
\def\LL{\mathcal L}
\def\SL{\LL^1}
\def\ZL{\LL^\infty}
\def\GG{\mathcal G}
\def\tt{\langle t\rangle}
\def\erf{\mathrm{Erf}}
\def\mgt#1{\textcolor{magenta}{#1}}
\def\ff{\rho}
\def\gg{G}
\def\sqrtnu{\sqrt{\nu}}
\def\ww{w}
\def\ft#1{#1_\xi}
\def\lec{\lesssim}
\def\gec{\gtrsim}
\renewcommand*{\Re}{\ensuremath{\mathrm{{\mathbb R}e\,}}}
\renewcommand*{\Im}{\ensuremath{\mathrm{{\mathbb I}m\,}}}
\ifnum\showllabel=1
 \def\llabel#1{\marginnote{\color{lightgray}\rm\small(#1)}[-0.0cm]\notag}
\else
 \def\llabel#1{\notag}
\fi

\newcommand{\eqnb}{\begin{equation}}
\newcommand{\eqne}{\end{equation}}

\newcommand{\norm}[1]{\left\|#1\right\|}
\newcommand{\nnorm}[1]{\lVert #1\rVert}
\newcommand{\abs}[1]{\left|#1\right|}
\newcommand{\NORM}[1]{|\!|\!| #1|\!|\!|}

\newtheorem{theorem}{Theorem}[section]
      \newtheorem{lemma}[theorem]{Lemma}
      \newtheorem{prop}[theorem]{Proposition}
      \newtheorem{corollary}[theorem]{Corollary}
      \newtheorem{definition}[theorem]{Definition}
      \newtheorem{examp}[theorem]{Example}
      \newtheorem{conj}[theorem]{Conjecture}
      \newtheorem{remark}[theorem]{Remark}

\def\theequation{\thesection.\arabic{equation}}
\numberwithin{equation}{section}
\definecolor{mygray}{rgb}{.6,.6,.6}
\definecolor{myblue}{rgb}{9, 0, 1}
\definecolor{colorforkeys}{rgb}{1.0,0.0,0.0}
\newlength\mytemplen
\newsavebox\mytempbox
\makeatletter
\newcommand\mybluebox{%
    \@ifnextchar[
       {\@mybluebox}%
       {\@mybluebox[0pt]}}
\def\@mybluebox[#1]{%
    \@ifnextchar[
       {\@@mybluebox[#1]}%
       {\@@mybluebox[#1][0pt]}}
\def\@@mybluebox[#1][#2]#3{
    \sbox\mytempbox{#3}%
    \mytemplen\ht\mytempbox
    \advance\mytemplen #1\relax
    \ht\mytempbox\mytemplen
    \mytemplen\dp\mytempbox
    \advance\mytemplen #2\relax
    \dp\mytempbox\mytemplen
    \colorbox{myblue}{\hspace{1em}\usebox{\mytempbox}\hspace{1em}}}
\makeatother
\def\bnew{\colr }
\def\enew{\colb }
\def\bold{\colu }
\def\eold{\colb }
\def\phiij{\phi_{ij}}
\def\un{u^{(n)}}
\def\Bn{B^{(n)}}
\def\unp{u^{(n+1)}}
\def\unm{u^{(n-1)}}
\def\Bnp{B^{(n+1)}}
\def\Bnm{B^{(n-1)}}
\def\pn{p^{(n)}}
\def\pnm{p^{(n-1)}}
\def\eeo{\tilde \epsilon}
\def\eet{\bar\epsilon}
\def\us{u_{\text S}}
\def\rr{r}
\def\weaks{\text{\,\,\,\,\,\,weakly-* in }}
\def\inn{\text{\,\,\,\,\,\,in }}
\def\cof{\mathop{\rm cof\,}\nolimits}
\def\Dn{\frac{\partial}{\partial N}}
\def\Dnn#1{\frac{\partial #1}{\partial N}}
\def\tdb{\tilde{b}}
\def\tda{b}
\def\qqq{u}
\def\lat{\Delta_2}
\def\biglinem{\vskip0.5truecm\par==========================\par\vskip0.5truecm}
\def\inon#1{\hbox{\ \ \ \ \ \ \ }\hbox{#1}}                
\def\onon#1{\inon{on~$#1$}}
\def\inin#1{\inon{in~$#1$}}
\def\FF{F}
\def\andand{\text{\indeq and\indeq}}
\def\ww{w(y)}
\def\ee{\mathrm{e}}
\def\startnewsection#1#2{\newpage \section{#1}\label{#2}\setcounter{equation}{0}}   
\def\nnewpage{ }
\def\sgn{\mathop{\rm sgn\,}\nolimits}    
\def\Tr{\mathop{\rm Tr}\nolimits}    
\def\div{\mathop{\rm div}\nolimits}
\def\curl{\mathop{\rm curl}\nolimits}
\def\dist{\mathop{\rm dist}\nolimits}  
\def\supp{\mathop{\rm supp}\nolimits}
\def\indeq{\quad{}}           
\def\period{.}                       
\def\semicolon{\,;}                  
\def\cmi#1{\underline{{\color{cyan} \underline{IK: #1}\colb}}}
\def\cmw#1{{\colr \bf WO: #1\colb}}
\def\colr{\color{red}}
\def\coli{\color{cyan}}
\def\colrr{\color{black}}
\def\colb{\color{black}}
\def\coly{\color{lightgray}}
\definecolor{colorgggg}{rgb}{0.1,0.5,0.3}
\definecolor{colorllll}{rgb}{0.0,0.7,0.0}
\definecolor{colorhhhh}{rgb}{0.3,0.75,0.4}
\definecolor{colorpppp}{rgb}{0.7,0.0,0.2}
\definecolor{coloroooo}{rgb}{0.45,0.0,0.0}
\definecolor{colorqqqq}{rgb}{0.1,0.7,0}
\def\colg{\color{colorgggg}}
\def\collg{\color{colorllll}}
\def\coleo{\color{colorpppp}}
\def\cole{\color{coloroooo}}
\def\cole{\color{black}}
\def\colu{\color{blue}}
\def\colc{\color{colorhhhh}}
\def\colW{\colb}   
\definecolor{coloraaaa}{rgb}{0.6,0.6,0.6}
\def\colw{\color{coloraaaa}}
\def\comma{ {\rm ,\qquad{}} }            
\def\commaone{ {\rm ,\quad{}} }          
\def\nts#1{{\color{red}\hbox{\bf ~#1~}}}
\def\ntsik#1{{\text{\color{red}IK:\hbox{\bf ~#1~}}}}
\def\ntsf#1{\footnote{\color{colorgggg}\hbox{#1}}} 
\def\blackdot{{\color{red}{\hskip-.0truecm\rule[-1mm]{4mm}{4mm}\hskip.2truecm}}\hskip-.3truecm}
\def\bluedot{{\color{blue}{\hskip-.0truecm\rule[-1mm]{4mm}{4mm}\hskip.2truecm}}\hskip-.3truecm}
\def\purpledot{{\color{colorpppp}{\hskip-.0truecm\rule[-1mm]{4mm}{4mm}\hskip.2truecm}}\hskip-.3truecm}
\def\greendot{{\color{colorgggg}{\hskip-.0truecm\rule[-1mm]{4mm}{4mm}\hskip.2truecm}}\hskip-.3truecm}
\def\cyandot{{\color{cyan}{\hskip-.0truecm\rule[-1mm]{4mm}{4mm}\hskip.2truecm}}\hskip-.3truecm}
\def\reddot{{\color{red}{\hskip-.0truecm\rule[-1mm]{4mm}{4mm}\hskip.2truecm}}\hskip-.3truecm}
\def\tdot{{\color{green}{\hskip-.0truecm\rule[-.5mm]{6mm}{3mm}\hskip.2truecm}}\hskip-.1truecm}
\def\gdot{\greendot}
\def\bdot{\bluedot}
\def\pdot{\purpledot}
\def\ydot{\cyandot}
\def\rdot{\cyandot}
\def\fractext#1#2{{#1}/{#2}}
\def\ii{\hat\imath}
\def\fei#1{\textcolor{blue}{#1}}
\def\vlad#1{\textcolor{cyan}{#1}}
\def\igor#1{\text{{\textcolor{colorqqqq}{#1}}}}
\def\igorf#1{\footnote{\text{{\textcolor{colorqqqq}{#1}}}}}
\def\Sl{S_{\text l}}
\def\Sll{\bar S_{\text l}}
\def\Sh{S_{\text h}}
\newcommand{\p}{\partial}
\newcommand{\wa}{\widetilde{\alpha}}
\newcommand{\os}{{\overline{S}}}
\newcommand{\oos}{{\overline{\overline{S}}}}
\newcommand{\low}{\mathrm{l}}
\newcommand{\high}{\mathrm{h}}
\newcommand{\UE}{U^{\rm E}}
\newcommand{\PE}{P^{\rm E}}
\newcommand{\KP}{K_{\rm P}}
\newcommand{\uNS}{u^{\rm NS}}
\newcommand{\vNS}{v^{\rm NS}}
\newcommand{\pNS}{p^{\rm NS}}
\newcommand{\omegaNS}{\omega^{\rm NS}}
\newcommand{\uE}{u^{\rm E}}
\newcommand{\vE}{v^{\rm E}}
\newcommand{\pE}{p^{\rm E}}
\newcommand{\omegaE}{\omega^{\rm E}}
\newcommand{\ua}{u_{\rm   a}}
\newcommand{\ou}{\overline{u}}
\newcommand{\oo}{\overline{\omega}}
\newcommand{\vn}{v^{(n)}}
\newcommand{\vnp}{v^{(n+1)}}
\newcommand{\vnm}{v^{(n-1)}}
\newcommand{\va}{v_{\rm   a}}
\newcommand{\omegaa}{\omega_{\rm   a}}
\newcommand{\ue}{u_{\rm   e}}
\newcommand{\ve}{v_{\rm   e}}
\newcommand{\omegae}{\omega_{\rm e}}
\newcommand{\omegaeic}{\omega_{{\rm e}0}}
\newcommand{\ueic}{u_{{\rm   e}0}}
\newcommand{\veic}{v_{{\rm   e}0}}
\newcommand{\vp}{v^{\rm P}}
\newcommand{\tup}{{\tilde u}^{\rm P}}
\newcommand{\bvp}{{\bar v}^{\rm P}}
\newcommand{\omegap}{\omega^{\rm P}}
\newcommand{\tomegap}{\tilde \omega^{\rm P}}
\renewcommand{\vp}{v^{\rm P}}
\renewcommand{\omegap}{\Omega^{\rm P}}
\renewcommand{\tomegap}{\omega^{\rm P}}
\renewcommand{\d}{\mathrm{d}}
\newcommand{\R}{\mathbb{R}}
\newcommand{\N}{\mathbb{N}}
\renewcommand{\tt}{{\widetilde{\tau}}}

\newcommand{\tX}{{\widetilde{X}}}
\newcommand{\tY}{{\widetilde{Y}}}
\newcommand{\oY}{{\overline{Y}}}

\begin{abstract}
We consider the incompressible Euler equations on an analytic domain $\Omega $ with nonhomogeneous boundary condition $u\cdot \mathsf{n} = \ou \cdot \mathsf{n}$ on $\p \Omega$, where $\ou$ is a given
 divergence-free 
analytic vector field. We establish local well-posedness for $u$ in analytic spaces without any compatibility conditions in
 all space dimensions. We also prove the global well-posedness in the $2$D case if $\overline{u}$ decays in time sufficiently fast.
\end{abstract}
\colb

\maketitle
\setcounter{tocdepth}{2}
\section{Introduction}\label{sec_intro}
In this paper, we are concerned with the inflow-outflow problem for
inviscid incompressible flows.
Namely, we fix an analytic domain
$\Omega \subset \R^d$ (see Section~\ref{sec_prelims} for the 
definition) and an analytic divergence-free  vector field
$\ou \colon [0,T_0] \times \Omega \to \RR^d $, where $T_0>0$, and we seek a
solution  to the Euler equations with the prescribed normal component of
the velocity on $\p \Omega$,
 \eqnb\label{intro_euler}
  \begin{split}
   \partial_{t} u + u\cdot \nabla u + \nabla p &= 0,\\
    \nabla \cdot u
    &=  0\hspace{0.45cm}    \inon{in $\Omega$}, 
    \\
    u\cdot \mathsf{n} &= \ou \cdot \mathsf{n}
    \inon{on $\partial\Omega$}
   ,
  \end{split}
  \eqne
where $\mathsf{n}$ represents an outward unit normal vector,
with the initial condition $u(0)=u_0$.  Formally,  the pressure $p$ in \eqref{intro_euler} satisfies the nonhomogeneous elliptic system
\eqnb\label{p_system}
\begin{split} -\Delta p &= \sum_{i,j} \p_i u_j \p_j u_i\qquad \text{ in } \Omega,\\
\p_{\mathsf{n}} p &= - (u\cdot \nabla  u )\cdot \mathsf{n}-\p_t \ou \cdot \mathsf{n} \qquad \text{ on } \p \Omega
.
\end{split}
\eqne
We note that the inflow-outflow problems arise in many physical situations.
Several examples include aerodynamics, where we consider the flow of air around a wing or a rocket, environmental flows, where we model flows of rivers into lakes, and blood flow, where  the inflow-outflow condition is time-dependent.
Also, if the domain containing a fluid is moving, we arrive naturally at the inflow-outflow problem~\cite[p.~5]{MP}.

The problem \eqref{intro_euler} was initially studied by Zaj\c{a}czkowski \cite{Z1,Z2,Z3,Z4}, Antontsev, Kazhikov, and Monakhov~\cite{AKM1,AKM2},
 then by Petcu in~\cite{P}, 
and more recently by Gie, Kelliher, and Mazzucato in~\cite{GKM1,GKM2}. It is known to become very difficult as soon as $\ou \cdot \mathsf{n} \ne 0$.
This can already be observed at the Sobolev level. Namely, set
  \begin{equation}
   v\coloneqq u - \ou
   ,
   \llabel{EQ11}
  \end{equation}
in which case the Euler equations \eqref{intro_euler} become
 \eqnb\label{intro_euler_v}
  \begin{split}
   \partial_{t} v + v\cdot \nabla v + \ou \cdot\na  v + v \cdot \na \ou + \ou \cdot\na  \ou  + \nabla p &= -\p_t \ou ,\\
    \nabla \cdot v
   & =  0
    \inon{in $\Omega$},
    \\
    v\cdot \mathsf{n} &= 0
    \inon{on $\partial\Omega$}
   ,
  \end{split}
  \eqne
with the initial condition $v(0)= v_0$. Given a multiindex $\alpha$, we apply $\p^\alpha$ to the Euler equations~\eqref{intro_euler_v}, multiply by $\p^\alpha v $, and integrate to obtain the standard estimate on the time evolution of $ \sum_{|\alpha|=k} \| \p^{\alpha } v\|_{L^2}$. One of the terms appearing on the right-hand side of such an estimate is 
\begin{equation}
\int_\Omega (\ou \cdot \na \p^\alpha v) \cdot \p^\alpha v
= \frac12 \int_{\p \Omega } | \p^\alpha v |^2 \ou \cdot \mathsf{n} ,
   \label{EQ17}
     \end{equation}
which  can be bounded by
\eqnb\label{der_loss}
\| \ou  \|_\infty \| \p^\alpha v \|_{L^2(\p \Omega )}^2  \lec \| \ou  \|_\infty
\| \p^\alpha v \|_{L^2(\Omega )}
\| \p^\alpha v \|_{H^{1}(\Omega )}
.
\eqne
The expression on the right-hand side is problematic since, to estimate the derivative of order $|\alpha |$, we need to control a derivative of the higher order~$|\alpha|+1$. 

To understand the problem better, we note that one may use the Gromeka-Lamb form of the Euler equations~\eqref{intro_euler} to show that the identities
\eqnb\label{cc}
\begin{split}
u^\mathsf{n} \omega^\mathsf{t} &= \left(-\p_t u^\mathsf{t} - \nabla_{\p \Omega } \left( p + \frac{|u|^2}2 \right)  \right)^\perp - u^\mathsf{t} \mathrm{curl}_{\partial \Omega } u^\mathsf{t} ,\\
\omega^\mathsf{n} &= \mathrm{curl}_{\p \Omega } u^\mathsf{t}
\end{split}
\eqne
must hold on $\p \Omega$, where $\nabla_{\p \Omega}$, $\mathrm{curl}_{\p \Omega }$ denote appropriate tangential derivatives on $\p \Omega$, and $v^\mathsf{n} \coloneqq v\cdot \mathsf{n}$, $v^\mathsf{t} \coloneqq v - v^\mathsf{n} \mathsf{n}$ denote the normal and tangential components, respectively, of a vector field $v$, and $v^\perp\coloneqq \mathsf{n} \times v$; see~\cite[Proposition~3.1]{GKM2} for details. Following observations in \cite{GKM2}, we note that \eqref{p_system} and \eqref{cc} appear to give, roughly speaking, 
\eqnb\label{issue}
\text{two relations among four quantities }\omega, \nabla_{\p \Omega}p, u^\mathsf{n} , u^\mathsf{t}, 
\eqne
which suggests that, in order to obtain any local well-posedness of \eqref{intro_euler} one must find, independently, a way of determining two relations, so that \eqref{p_system} and \eqref{cc} provide the remaining two.
This is related to the fact that, on the ``inflow portion'' of $\p \Omega$, the incoming vorticity $\omega$ must lie in the range of $\mathrm{curl}$ and that $\mathrm{div}\,\omega$ must remain conserved (see \cite[(1.6)]{GKM2}). One way of dealing with this issue is to impose an additional boundary condition on the entire velocity field $u$ on the inflow portion $\p \Omega_{\rm in}$ of the boundary~$\p \Omega$. This problem was considered by \cite[Theorem~1.2]{GKM2}, who proved the local well-posedness in $C^{N+1,\alpha}$ ($\alpha \in (0,1)$) in the case of unsteady $\ou$, provided a certain compatibility condition, involving the $N$-th and ($N+1$)-st time derivatives, holds on $\p \Omega_{\rm in}$ at~$t=0$. 

Another way of dealing with the above issue is to prescribe $\omega$ on $\p \Omega_{\rm in}$---the so-called \emph{vorticity boundary conditions}; see~\cite{Z1} as well as ~\cite[Theorem~1.3]{GKM2}. We also refer the reader to \cite{Z2} where a boundary condition for pressure is assumed on the outflow part of the boundary, to \cite{Z2,Z3,Z4} for results concerned with domains with corners, to \cite{GKM2} for an extensive introduction to the problem, as well as to the classical work \cite{AKM2} on the subject.
We point out that the existence and uniqueness problem when the vorticity is prescribed at the
inflow portion of the inflow-outflow Euler problem remains open and
not well-understood. Also, the problem increases in complexity in higher 
regularity Sobolev spaces.

In this paper, we prove the unique solvability of the
inflow-outflow problem when the data are analytic,
however without any additional condition to \eqref{intro_euler}$_3$. This is possible since analytic data can be uniquely extended to a neighborhood of $\Omega$, and thus the compatibility conditions can be bypassed for a short time.
We emphasize that the inflow and outflow regions do not need to be
separated, and they can even move over time as the function $\bar u$ in \eqref{intro_euler}$_3$
is time-dependent.
Our approach works equally in all space dimensions, two or
higher, although we present full details in the space dimension two,
pointing out the adjustments for the dimensions three or higher. 

To describe our approach, consider the analytic norms
 \begin{align}
  \begin{split}
   &
   \Vert u\Vert_{X(\tau)}
   \coloneqq 
  \sum_{i+j\geq r } c_{ij} 
    \Vert \p^i T^j  u\Vert , \qquad \text{ where } c_{ij} \coloneqq \frac{(i+j)^r}{(i+j)!}\tau^{i+j-r }
  \overline{\epsilon}^i \epsilon^j,
    \\&
   \Vert u\Vert_{\tX(\tau)}
   \coloneqq 
   \Vert u\Vert_{X(\tau)}
   + \Vert u\Vert_{H^{r}}
  \\&
   \Vert u\Vert_{Y(\tau)}
   \coloneqq 
   \sum_{ i+j\geq r+1}
    \frac{i+j}{
     \tau
    }c_{ij}
    \Vert \p^i T^j  u\Vert
    \\&
   \Vert u\Vert_{\widetilde{Y}(\tau)}
   \coloneqq 
  \tau  \Vert u\Vert_{Y(\tau)}
   + \Vert u\Vert_{H^{r}} \\&
   \Vert u\Vert_{\oY (\tau)}
   \coloneqq 
  \Vert u\Vert_{Y(\tau)}
   + \Vert u\Vert_{H^{r}},
  \end{split}
   \label{EQ19c}
  \end{align}
  where $r\geq 3$ is fixed, and where we use the notation
\begin{equation}
\| \cdot \| \coloneqq \| \cdot \|_{L^2 (\Omega )}
   \label{EQ18}
     \end{equation}
for the $L^2$ norm. Here, $\tau (t)$ denotes a time-dependent analyticity radius.  The choice of $\tau (t)$ is the main degree of freedom in the analytic approach, and an appropriate choice of $\tau(t)$ provides us with the \emph{analytic~dissipation}. This in turn allows us to overcome a derivative loss; see~\eqref{002} and the discussion that follows. 
If the choice of the analyticity radius $\tau$ is clear from the context, we use a short-hand notation
\begin{equation}
\| \cdot \|_{\widetilde{X} } \equiv \| \cdot \|_{\widetilde{X}(\tau )},\qquad \| \cdot \|_{Y } \equiv \| \cdot \|_{{Y}(\tau )},\qquad \| \cdot \|_{\overline{Y}} \equiv  \| \cdot \|_{Y(\tau )} + \| \cdot \|_{H^r}
.
   \label{EQ20}
     \end{equation}

The tangential operator $T$ is a technical device required to handle the case of a nonflat boundary $\p \Omega$; we refer the reader to Section~\ref{sec3} for the definition of tangential operators $T$ and for the notation~$\Vert \p^i T^j  u\Vert$. The constants $\epsilon$, $\overline{\epsilon }>0$ are used for absorbing lower order commutator terms as well as tangential reduction arguments to transfer normal derivatives to tangential ones; see \eqref{transfer}, for example.  We suppose that ${\epsilon}$ and $\overline{\epsilon}/\epsilon$ are sufficiently small positive numbers, as required by Lemma~\ref{L_pressure_curved} below. The number $\overline{\epsilon} /\epsilon$ is related to the quotient of the normal and tangential analyticity radii, which we require to be bounded by a small constant. We discuss the spaces $X,\tX,Y,\tY,\oY$ in more detail below Theorem~\ref{T01}, but we note at this point that they were introduced in \cite{CKV,KV1,KV2}; see also  \cite{K1,K2} for the non-summed norms and the Navier-Stokes setting. We also note that
  \eqnb \label{grad_in_X_is_Y}
  \|\na u \|_{\tX } \lec \| u \|_{\oY}.
  \eqne
We now define the notion of a solution in analytic spaces.

\begin{definition}
\label{D_sol}
{\rm
Given $T_0 >0$ and $\tau\colon[0,T_0 ]\to(0,\infty)$, a function
$v\colon[0,T_0 ]\to \tX(\inf_{t\in[0,T_0]}\tau(t))$  is a solution of the system \eqref{intro_euler_v}
on some interval $[0,T_0]$  if:\\
(i) $v\in L^{\infty}([0,T_0],\tX(\tau))\cap L^{1}([0,T_0],\tY(\tau))$, 
\\
(ii) there exists $p\in L^{1}([0,T_0],\tX(\tau))$
such that $\nabla p\in L^{1}([0,T_0],\tX(\tau))$ and
  \begin{equation}
   v(t)
   = u_0 - \ou (t)
    - \int_{0}^{t} (v\nabla v + \ou \cdot\na  v + v \cdot \na \ou + \ou \cdot\na  \ou  + \nabla p
                     - \partial_{t}\ou
                   ) \,\d s
\comma
t\in [0,T_0],
   \label{EQ01}
  \end{equation}
(iii) the second and third equations in \eqref{intro_euler_v} hold.
}
\end{definition}
We assume that the background velocity field $\ou$ satisfies
\eqnb\label{ou_requirements}
\| \ou \|_{\tX(\tau_0)} + \| \ou \|_{\oY (\tau_0)} +
\| \p_t \ou \|_{\tX (\tau_0)}<\infty
,
\eqne
for some $\tau_0\in(0,1]$.

Throughout, we assume, 
without loss of generality,
that $\int_\Omega p =0$ holds for all times. 
The following is our main result on the local well-posedness of the inflow-outflow problem.

\cole
\begin{theorem}[Local well-posedness of \eqref{intro_euler_v} in analytic spaces]\label{T01}
Let $r= d+1 $. Assume that $\Omega \subset \R^d$, where $d\geq 2$, is an analytic domain, and suppose that $v_0\in X(\tau_0)$, for some $\tau_0\in(0,1]$. Then there exists $M\geq 1$  and a unique solution $v\in C ([0,T_0];\widetilde{X}(\tau(t)) )$ to \eqref{intro_euler_v}, satisfying the initial condition $v(0)=v_0$ and the estimate
\eqnb\label{002}
\frac{\d }{\d t } \| v \|_{\widetilde{X}(\tau )} - \dot \tau \| v \|_{Y(\tau )} \lec_r \| v \|_{\widetilde{X}(\tau )} \| v \|_{\widetilde{Y}(\tau )} + \| \ou \|_{\overline{Y}(\tau )} \| v \|_{\widetilde{X}(\tau )} + \| \ou \|_{\widetilde{X}(\tau )} \| v \|_{\overline{Y}(\tau )}+ \| \ou \|_{\widetilde{X}(\tau )} \| \ou \|_{\overline{Y}(\tau )} + \| \p_t \ou \|_{\tX (\tau )}
,
\eqne
for all $t\in [0,T_0]$, where 
\eqnb\label{def_tau}
\tau (t) \coloneqq \tau_0 - M t,
\eqne
and $T_0\coloneqq \tau_0 /M$.
\end{theorem}
\colb

We emphasize that Theorem~\ref{T01} does not involve \emph{any compatibility conditions} (except for the divergence-free condition $\mathrm{div}\,\ou=0$, which in turn guarantees that
$\int_{\p \Omega } \ou \cdot \mathsf{n} =0$. 
This does not contradict the issue \eqref{issue} discussed above since it is not clear whether the quantities $\omega$ , $\nabla_{\p \Omega } p$, $u^\mathsf{n}$, and $u^\mathsf{t}$ are entirely independent of each other. In fact, Theorem~\ref{T01} shows that they are not in the analytic setting.

On the other hand, the issue of the derivative loss pointed out in \eqref{der_loss} is accommodated by the analytic dissipation. Roughly speaking, $Y$ encodes the derivative loss through the second term on the left-hand side of~\eqref{002}. In order to illustrate this, suppose that there exists $A>0$, depending only on  $\ou$ and $v_0$, such that 
\eqnb\label{002_spse}
\sup_{t\in [0,T_0]} \| v \|_{\tX } + M \int_0^{T_0} \| v \|_{Y } \leq A.
\eqne
For illustration purposes, consider a simplified version of \eqref{002}, where we only take into account the first term on the right-hand side
of \eqref{002}
with the norm $\| \cdot \|_{\tY}$ replaced by its analytic part $\| \cdot \|_{Y}$, 
\eqnb\label{002_spse1}
\frac{\d }{\d t} \| u \|_{\tX } + M \| v \|_{Y}  \leq C_0 \| v \|_{\tX } \| v \|_{Y}. 
\eqne
This way we consider only the leading order quadratic term on the right-hand side, demonstrating the strength of the analytic dissipation. Note that \eqref{002_spse1} gives 
\[ \sup_{t\in [0,T_0]} \| v \|_{\tX } + M \int_0^{T_0} \| v \|_{Y } \leq C_0 \int_0^{T_0} \| v \|_{\tX } \| v \|_{Y} \leq C_0 \sup_{t\in [0,T_0]} \| v \|_{\tX } \int_0^{T_0} \| v \|_{Y} \leq C_0 A^2 M^{-1},
\]
where we used the assumption \eqref{002_spse} in the last inequality.
Thus, choosing a sufficiently large $M\geq 1$, the right-hand side is bounded by $A/2$, which is consistent with \eqref{002_spse} and also gives the leeway of another $A/2$ for estimating the remaining terms appearing on the right-hand side of~\eqref{002}. In particular, the Sobolev part of $\| v \|_{\tY}$, which we omitted above,  can be bounded by~$\| v \|_{\tX}$. It is thus clear that, upon making an appropriate choice of $A$ (see~\eqref{def_A} below) and sufficiently large $M\geq 1$ (see~\eqref{choice_M_T0} below), the local well-posedness result of Theorem~\ref{T01} can be obtained with $T_0=\tau_0 /M$. Note that a precise argument requires a Picard-type argument; see Section~\ref{sec_flat_conclude} for details. We emphasize that the local-posedness result is possible by ensuring that the analyticity radius decreases sufficiently quickly, i.e., we set the analytic dissipation $M\geq 1$ sufficiently large.

We emphasize that Theorem~\ref{T01} does not take advantage of the fact that the analytic part of $\| v \|_{\tY}$ on the right-hand side of \eqref{002} involves an additional factor of~$\tau$. This is because we need to take $M\geq 1$ large, depending on~$\ou$. Moreover, in the $3D$ case the time of existence cannot be extended further not only due to $\ou$, which could be large, but also because the $3$D Euler equations are only locally well-posed. In contrast, in the 2D case, we can explore the global Sobolev existence by and assuming that $\ou$ decays sufficiently fast and choosing $\dot \tau$ proportional to $-\tau$, since this is what making the best use of ``$\tau$'' in $\| v \|_{\tY} = \tau \| v \|_{Y} + \| v \|_{H^r}$, appearing in the first term on the right-hand side of \eqref{002}, suggests. This way, we obtain the global well-posedness (see Theorem~\ref{T02}),
whose proof takes advantage of the
\emph{persistence of regularity}; see the comments below Theorem~\ref{T02}.

 Before stating the global well-posedness result in the 2D case, we note that the condition $r=d+1$ in Theorem~\ref{T01} (which is sufficient but not optimal) is related to the product estimates for the lower order part of the analytic norm $\|\cdot \|_{\tX}$, namely the part that does not admit the analytic dissipation. To be precise, we need the condition $r\geq d+1$ for the product estimates (Lemma~\ref{L_product}, Corollary~\ref{cor_products}, \eqref{product_curved}, \eqref{010_curved}--\eqref{013_curved}), as it plays a crucial role in the factorial bounds (see \eqref{calc_al}--\eqref{calc_ah}, for example) and in controlling the powers of $\tau$ (see \eqref{bl}--\eqref{bh}, for example). It is also required in performing the {persistence of regularity} argument (see the comments below Theorem~\ref{T02}) and in the pressure estimates (see Lemmas~\ref{L02} and~\ref{L_pressure_curved}). The choice $r=d+1$ is sufficient for our purposes, and we fix this value for $r$ throughout, but we note that it can be relaxed to any $r>d/2$, as one can verify directly. 

We provide details for the case $d=2$ of Theorem~\ref{T01}; for other $d$, the arguments apply analogously, except that the computations become more complicated. The computations in the case $d=2$ should convince the reader that the relevant cancellations and summations properties persist in any space dimension.

We note that the treatment of an arbitrary domain $\Omega$ requires some technical computations regarding the tangential operator $T$, and so, for the sake of simplicity, we first focus on the \emph{flat case}, namely  
\begin{equation}
\Omega = \TT \times (0,1),
   \label{EQ21}
     \end{equation}
where $\TT$ denotes the flat torus. This already demonstrates the strength of the method and reveals the main idea of the proof of Theorem~\ref{T01}, 
with a general analytic domain treated further below.
 The main observation in the flat case is that the partial derivative $\p_{1}$ is tangential to the boundary. Moreover, $\p_{2}$
is normal to the boundary, and the two derivatives together form the full gradient. Consequently,  the product and pressure estimates are simpler, and therefore it is reasonable to replace the analytic norms~\eqref{EQ19c} with simpler definitions
 \begin{align}
  \begin{split}
   &
   \Vert u\Vert_{X(\tau)}
   \coloneqq 
   \sum_{ |\alpha |\geq r}
    \frac{
     |\alpha |^{r}
        }{
     |\alpha |!
    }
    \tau^{|\alpha |-r}
   \epsilon^{\alpha_2}
    \Vert \p^\alpha u\Vert
    ,
    \\&
   \Vert u\Vert_{\tX(\tau)}
   \coloneqq 
   \Vert u\Vert_{X(\tau)}
   + \Vert u\Vert_{H^{r}}
  ,
  \\&
   \Vert u\Vert_{Y(\tau)}
   \coloneqq 
   \sum_{ |\alpha |\geq r+1}
    \frac{
     |\alpha |^{r+1}
        }{
     |\alpha |!
    }
    \tau^{|\alpha |-r-1}
    \epsilon^{\alpha_2}
    \Vert\p^\alpha u\Vert
    ,
    \\&
   \Vert u\Vert_{\widetilde{ Y}(\tau)}
   \coloneqq 
  \tau  \Vert u\Vert_{Y(\tau)}
   + \Vert u\Vert_{H^{r}}
   ,
    \\&
   \Vert u\Vert_{\overline{Y}(\tau)}
   \coloneqq 
  \Vert u\Vert_{Y(\tau)}
   + \Vert u\Vert_{H^{r}}.
  \end{split}
   \label{EQ19}
  \end{align}

As for the case of a general domain, treated in Section~\ref{sec3}, we introduce the necessary concepts in Section~\ref{sec_prelims}, and we prove Theorem~\ref{T01} in full generality in Section~\ref{sec_curved}.  The proof requires introducing an abstract notion of a tangential operator~$T$ and taking into account the resulting commutators.
It also requires the derivative reductions (Lemma~\ref{L001}) to get a sufficiently strong estimate on the pressure, particularly regarding the boundary condition in \eqref{EQ25_curved}; see Section~\ref{sec_pressure_gen} for details.

Our second result asserts the global-in-time well-posedness in the 2D case under a decay-in-time assumption on~$\ou$.

\begin{theorem}[Global well-posedness in $2$D]\label{T02}
Let $d=2$ and $\tau_0\in(0,1]$, and suppose that $v_0\in \widetilde{X}(\tau_0 )\cap \overline{Y}(\tau_0 )$. Then there exist $A,B\geq1$ such that if
\begin{equation}
\tau (t) \coloneqq \tau_0 \exp\left( -A \int_0^t \exp \exp \exp (A s)\d s \right)
   \label{EQ22}
     \end{equation}
and
 $\| \ou (t) \|_{\widetilde{X}(\tau(t))}+ \| \ou (t) \|_{\overline{Y}(\tau(t))}+\| \p_t \ou (t) \|_{\tX (\tau (t))}  \leq f(t)$ for all $t\geq 0$, where  $f\colon [0,\infty )\to [0,\infty )$ satisfies
 \begin{equation}
 B f(t) \exp \exp \exp \exp (Bt)\leq 1, \qquad \text{ for all }t\geq 0,
   \label{EQ23}
     \end{equation}
 then there exists a unique solution $v\in L^\infty ([0,\infty);\widetilde{X}(\tau))$ to the system \eqref{intro_euler_v} with $-\dot \tau v \in L^1 ((0,\infty ); \overline{Y}(\tau ))$. 
\end{theorem}
Here, $v$ is a solution to \eqref{intro_euler_v} in the sense that Definition~\ref{D_sol} holds for any $T_0>0$.

The proof of the above theorem uses the \emph{persistence of analyticity}, namely that the existence in analytic spaces can be extended as long as a certain subcritical norm remains under control. This can be observed in the a~priori estimate~\eqref{002}, which shows that the only quadratic term on the right-hand side, namely  $\| v \|_{\widetilde{X}(\tau )} \| v \|_{\widetilde{Y}(\tau )}$, involves an additional factor of $\tau$ next to the $Y$ norm (recall \eqref{EQ19c}, \eqref{EQ19}). This shows that if $\tau\colon [0,\infty )\to [0,\infty )$ satisfies 
\eqnb\label{temp23}
-\dot \tau \geq C \left( \tau\| v \|_{\widetilde{X}} +f \right)
 \eqne
  with a large $C\geq1$,  then the $Y$ norm can be absorbed by the left-hand side, and thus the right-hand side of \eqref{002} grows linearly in $v$ provided $\| v \|_{H^3}$ is bounded by a constant. However, due to the inflow-outflow velocity $\ou$, it is not immediately clear how to control~$\| v \|_{H^3}$. Indeed, noting that the vorticity $\theta \coloneqq \mathrm{curl}\, (v+\ou )$ satisfies 
 \begin{equation}
 \p_t \theta + (v+\ou ) \cdot \nabla \theta =0,
   \label{EQ24}
     \end{equation}
we observe that $\| \theta \|_{L^\infty}$ is \emph{not conserved} as $\ou$ can bring nonzero vorticity from outside of~$\Omega$. Thus it is not clear, at this point, how to obtain the global well-posedness. However, one may observe that, since $v_0$  is analytic, it can be uniquely extended to a small neighborhood of $\Omega$, which gives us control of the vorticity that can be brought in by $\ou$, provided $\ou$ decays sufficiently fast (as mentioned above). To be precise, observing  that 
 \begin{equation}
 \frac{\d }{\d t} \| \theta \|_\infty \leq \| \ou \|_\infty \| \nabla \theta \|_\infty\leq \| \ou \|_\infty (\| v \|_{W^{2,\infty}} + \| \ou \|_{W^{2,\infty }} ) ,
   \label{EQ26}
     \end{equation}
we can estimate $\| v \|_{W^{2,\infty}}$ by $\| v \|_{\tX (\tau )}$. This lets us predict, for sufficiently decaying $\ou$, an estimate on the global-in-time growth of $\| v \|_{\tX (\tau )}$ as well as the decay of $\tau(t)$ that both control $\| \theta \|_\infty$ globally, and hence also $\| v \|_{H^3}$ and~$\| v \|_{\tX (\tau )}$. Thus we can back-define $\tau$, using the smallness assumption \eqref{EQ23} on $f$ and a continuity argument; see Section~\ref{sec_global} for details.

\subsection{Smoothness of solutions}
We note that Theorem~\ref{T01} only gives $u$ which is bounded with values in $X(\tau)$, and it is not immediately clear that $u$ is a smooth function of~$t$ if $\ou$ is. To explain this, assume that $\ou$ is smooth in space and time. We 
first note that \eqref{EQ01} holds as an equality in $H^k$ for every $k\geq 0$, and so
  \begin{equation}
   u(t_2)-u(t_1)
   = - \int_{t_1}^{t_2}
       (u\nabla u + \nabla p)
     \,\d s
 \quad \text{ for } 0\leq t_1<t_2\leq T
     ,
   \label{EQ02}
  \end{equation}
  as an equality in $H^k$ for every $k\geq 0$. Varying $t_1,t_2$, we obtain that $u\in C([0,T_0]; H^k)$ for every $k\geq 0$. Then, applying the divergence operator to \eqref{EQ02}, varying
$t_1$ and $t_2$, and applying Lebesgue Differentiation Theorem and trace estimates shows that \eqref{p_system} holds for almost every $t\in [0,T_0]$. Modifying $p$ on a set of times of measure $0$, we thus obtain that $p\in C([0,T_0];H^k)$ for every $k\geq 0$. Dividing \eqref{EQ02} by $t_2-t_1$ and varying $t_1$, $t_2$, we obtain that $u\in C^1 ([0,T_0];H^k)$ for every $k\geq 0$ and consequently $p\in C^1([0,T_0];H^k)$ for every $k\geq 0$. We can continue by induction, obtaining that $u$ and $p$
have the same degree of smoothness as~$\ou$.

\section{The case of a flat domain}\label{sec_flat_LWP}

In this section, we prove Theorem~\ref{T01} for the case of a periodic channel, i.e., $\Omega =\TT \times (0,1)$. 

\subsection{A~priori analytic estimate (periodic channel)}\label{sec_apriori_analytic_flat}

Here, we derive the analytic a~priori bound  \eqref{002}, provided that the analyticity radius $\tau  \colon [0,T] \to \R_+$ and the solution $v \colon [0,T]\times  \Omega \to \R^2$ are given; the construction is presented in Section~\ref{sec_flat_conclude}.
We apply $\p^\alpha$ to the first equation in \eqref{intro_euler_v} and  test with $\p^\alpha v$ obtaining
  \eqnb\label{EQ08}
  \begin{split}
   \frac12
   \frac{\d}{\d t}
   \int
  |\partial^{\alpha}v|^2
     =&-
         \int ( (v+\ou)\cdot \nabla  \partial^{\alpha} v) \cdot \partial^{\alpha}v
       - \int S_\alpha (v+\ou  ,v+\ou )\cdot \p^\alpha v   \\
       &-\int \left( (v+\ou ) \cdot \nabla \p^\alpha \ou  \right)\cdot \p^\alpha v  -
   \int \partial^{\alpha} \nabla p \cdot \partial^{\alpha} v-\int \p^\alpha \p_t \ou \cdot \p^\alpha v
   ,
  \end{split}
  \eqne
  where we set
\eqnb\label{def_Salpha}
S_\alpha (v,w) \coloneqq \sum_{0<\beta \leq \alpha } {\alpha \choose \beta } (\p^\beta v \cdot \nabla ) \p^{\alpha-\beta } w.
\eqne  
Here and in the sequel, we use the notation $\int\coloneqq \int_\Omega$.
Note that the first term on the right-hand side of~\eqref{EQ08} may be estimated by $\| \ou \|_\infty \| \p^\alpha v \| \| \nabla \p^\alpha v \|$. 

The main difficulties in obtaining the a~priori estimate \eqref{002} is the treatment of the nonlinear terms $\|S(\cdot , \cdot )\|$ and the pressure term~$\| \p^\alpha \nabla p \|$. 

\cole
\begin{lemma}(product estimate in the case of periodic channel)\label{L_product}
{With $r\geq d+1$}, 
we have 
\begin{equation}
\sum_{|\alpha |\geq r } \frac{|\alpha |^r}{|\alpha |!}\tau^{|\alpha |-r } \epsilon^{\alpha_2} \| S_\alpha (u,v)\|  \lec_r \| v \|_{\widetilde{Y}(\tau )} \| u \|_{\widetilde{X}(\tau )} +\| v \|_{\widetilde{X}(\tau )} \| u \|_{\widetilde{Y}(\tau )}.
   \label{EQ28}
     \end{equation}
\end{lemma}
\colb

\begin{proof}[Proof of Lemma~\ref{L_product}]
Assume that $\tau$ is smooth, positive and decreasing on
$[0,T]$ with $\tau(0)=\tau_0\leq 1$.
First, we apply H\"older's inequality to~$\| S_\alpha (u,v) \|$.
We use the $L^\infty$ norm for the derivatives of smaller order, and $L^2$ for the larger order. Namely, for  $|\beta |\leq |\alpha | /2$, we employ the estimate 
\begin{equation} \| (\p^\beta u \cdot \nabla ) \p^{\alpha -\beta } v \| \leq \| \p^\beta u \|_{L^\infty} \|  \nabla  \p^{\alpha -\beta } v \|
,
   \label{EQ29}
     \end{equation}
and for $|\beta |> |\alpha | /2$ we switch the $L^\infty$ and $L^2$ norms. This gives 
\begin{equation}
\| S_\alpha (u,v)\| \lec  \sum_{0< |\beta |\leq |\alpha |/2 } {\alpha\choose\beta}
     \| \partial^{\beta}u\|_{L^{\infty}} \|\partial^{\alpha-\beta}\nabla v\|+  \sum_{ |\beta |> |\alpha |/2 } {\alpha\choose\beta}
     \| \partial^{\beta}u\| \|\partial^{\alpha-\beta}\nabla v\|_{L^{\infty}}
,
   \label{EQ30}
     \end{equation}
for every $|\alpha |\geq r$. We now  interpolate the $L^\infty$ norm in each of the two sums by $L^2$ and~$H^2$. Namely, we use 
\eqnb\label{interp1}
\begin{split}
\| \p^\beta u \|_{L^\infty } &\lec \| \p^\beta u \|^{1/2} \| D^2 \p^\beta u \|^{1/2} + \| \p^\beta u \|,
\\
\| \p^{ \alpha - \beta }\na  v \|_{L^\infty } &\lec \| \p^{\alpha - \beta } \na v \|^{1/2} \| D^2 \p^{\alpha - \beta}\na v  \|^{1/2} + \| \p^{\alpha -\beta  } \na v \|
\end{split}
\eqne
for the first and second sums, respectively. 
Using the inequality ${\alpha \choose \beta }\leq {|\alpha | \choose |\beta |}$ (see~\cite[p.~35]{LM01}) and the notation
\begin{equation}\begin{split}
F_\alpha (u )& \coloneqq \frac{\epsilon^{\alpha_2 1_{|\alpha |\geq r }} |\alpha |^r \tau^{(|\alpha | -r ) \vee 0}}{|\alpha |!} \| \p^\alpha u \| ,\\
 G_\alpha (u) &\coloneqq \frac{\epsilon^{\alpha_2 1_{|\alpha |\geq r +1}} |\alpha |^{r+1} \tau^{(|\alpha | -r -1) \vee 0}}{|\alpha |!} \| \p^\alpha u \|,
   \label{EQ31}\end{split}
     \end{equation}
we thus obtain
\eqnb\label{004}
\begin{split}
\sum_{|\alpha |\geq r } \frac{|\alpha |^r}{|\alpha |!}\tau^{|\alpha |-r } \epsilon^{\alpha_2} \| S_\alpha (u,v) \| &\lec \sum_{\substack{|\gamma | \leq 2 \\ |\kappa | =1}}  \sum_{|\alpha |\geq r }\biggl(  \sum_{0< |\beta |\leq |\alpha |/2 } a_{\low } b_{\low } F_{\beta }(u)^{1/2} F_{\beta+\gamma }(u  )^{1/2} G_{\alpha-\beta +\kappa } (v)  \\
&\hspace{2cm} +\sum_{ |\beta |> |\alpha |/2 } a_{\high } b_{\high } G_\beta (u) F_{\alpha -\beta +\kappa } (v)^{1/2} F_{\alpha - \beta + \kappa +\gamma }(v)^{1/2} \biggr),
\end{split}
\eqne
where we used
\eqnb\label{epsilon_est}
\epsilon^{\alpha_2} \epsilon^{-\frac12 \beta_2 1_{|\beta |\geq r}}\epsilon^{-\frac12 (\beta_2+\gamma_2 ) 1_{|\beta +\gamma  |\geq r}} \epsilon^{- (\alpha_2-\beta_2+\kappa_2) 1_{|\alpha -\beta +\kappa  |\geq r+1}} \leq \epsilon^{-\frac12 \gamma_2 - \kappa_2} \leq \epsilon^{-2} \lec 1
\eqne
in the case $|\beta |\leq |\alpha |/2$ (employing the notation $\kappa = (\kappa_1,\kappa_2)$, and similarly for the mutiindices $\alpha, \beta, \gamma$), and analogously for $|\beta | > |\alpha |/2$. We also used the fact that the low coefficient $a_{\low}$ of factorials, for $|\kappa |=1$, $|\gamma |\leq 2$,
is bounded as
\eqnb\label{calc_al}\begin{split}
a_{\low} &\coloneqq  \frac{|\alpha |^r }{|\beta |! |\alpha - \beta |! } \cdot \left( \frac{|\beta |! |\beta + \gamma |!}{|\beta |^r |\beta + \gamma |^r } \right)^{1/2}\cdot \frac{|\alpha - \beta + \kappa |! }{|\alpha - \beta +\kappa |^{r+1}}\\
&= \underbrace{\frac{(|\alpha - \beta + \kappa |-1)!}{|\alpha - \beta |!}}_{\lec 1} \left( \frac{|\beta +\gamma |!}{|\beta|! |\beta|^r |\beta +\gamma |^r} \right)^{1/2}  \underbrace{ \frac{|\alpha |^r}{ |\alpha - \beta +\kappa |^{r} } }_{\lec 1}
\\&
\lec_r
 \left( \frac{|\beta +\gamma |!}{|\beta|! |\beta|^r |\beta +\gamma |^r} \right)^{1/2}
\lec \left(\frac{|\beta+\gamma|!}{|\beta |! |\beta|^{2r} }\right)^{1/2}
\lec 1,
\end{split}
\eqne
as $r\geq 2$.
For the high factorial coefficient $a_{\high}$, we write
\eqnb\label{calc_ah}
\begin{split}
 a_{\high } &\coloneqq  \frac{|\alpha |^r }{|\beta |! |\alpha - \beta |! } \cdot  \frac{|\beta |! }{|\beta |^{r+1}} \cdot \left(\frac{|\alpha - \beta + \kappa |! |\alpha - \beta + \kappa +\gamma  |! }{|\alpha-\beta + \kappa  |^r |\alpha-\beta + \kappa+\gamma  |^r } \right)^{1/2}   \\
 &=  \underbrace{\frac{(|\alpha - \beta + \kappa |-1)!^{1/2}(|\alpha - \beta + \kappa +\gamma |-3)!^{1/2} }{|\alpha - \beta |!}}_{\lec 1}\\
&\hspace{1cm}\cdot \left( \frac{|\alpha -\beta + \kappa |\,\,|\alpha -\beta + \kappa +\gamma |! }{ (|\alpha -\beta + \kappa +\gamma| -3 )!| \alpha - \beta +\kappa |^r  | \alpha - \beta +\kappa +\gamma |^{r}} \right)^{1/2}  \underbrace{ \frac{|\alpha |^r}{ |\beta |^{r+1} } }_{\lec 1}  \\
&
\lec_r  \left( \frac{(|\alpha -\beta|+1)(|\alpha -\beta| + 1 + |\gamma|)^{3}}{( | \alpha - \beta| + 1 )^r  (| \alpha - \beta| + 1 + |\gamma |)^{r} }  \right)^{1/2}
\lec_r 1
.
\end{split}
\eqne
Moreover, in \eqref{004} we denoted the low and high coefficients of powers of $\tau$ by 
\begin{equation}
\begin{split}
b_{\low } &\coloneqq \frac{\tau^{|\alpha |-r}}{\tau^{((|\beta |-r)\vee 0)/2}\tau^{((|\beta+\gamma |-r)\vee 0)/2} \tau^{(|\alpha - \beta + \kappa |-r-1)\vee 0} },\\
b_{\high } &\coloneqq \frac{\tau^{|\alpha |-r}}{\tau^{(|\beta |-r-1)\vee 0}\tau^{((|\alpha - \beta+\kappa  |-r)\vee 0)/2} \tau^{((|\alpha - \beta + \kappa +\gamma |-r)\vee 0)/2} }. 
\end{split}
   \label{EQ35}
     \end{equation}
Recalling from \eqref{EQ19} that $\| u \|_{\widetilde{X}(\tau )} = \sum_{|\beta |\geq 0 } F_\beta (u)$ and $\| u \|_{{Y}(\tau )} = \sum_{|\beta |\geq r+1 } G_\beta (u)$, we now separate  the low and the high sums into two parts in a way that one is bounded by $\tau \| \cdot  \|_{Y(\tau )}$ and the other by $\| \cdot \|_{H^r}$, where the function under the norm is either $u$ or~$v$. To be more precise, we show below that
\begin{eqnarray}
b_{\low}  &\lec \begin{cases} 
\tau \qquad &\text{ if } |\alpha - \beta |\geq r, \\
1 &\text{ if } |\alpha - \beta | \leq r-1,
\end{cases} \label{bl}\\
b_{\high } &\lec \begin{cases} 
\tau \qquad &\text{ if }|\beta |\geq r+1, \\
1 &\text{ if } |\beta |\leq r.
\end{cases}\hspace{0.8cm} \label{bh}
\end{eqnarray}
Supposing that \eqref{bl}--\eqref{bh} hold, \eqref{004} gives
\eqnb\label{calc001}
\begin{split}
\sum_{|\alpha |\geq r } \frac{|\alpha |^r}{|\alpha |!}\tau^{|\alpha |-r } \epsilon^{\alpha_2} S_\alpha (u,v) &\lec \sum_{\substack{|\gamma | =2 \\ |\kappa | =1}}  \sum_{|\alpha |\geq r }\left(  \left( \sum_{\substack{0< |\beta |\leq |\alpha |/2 \\ |\alpha - \beta |\leq r-1}} + \tau \sum_{\substack{0< |\beta |\leq |\alpha |/2 \\ |\alpha - \beta |\geq r}}   \right) F_{\beta }(u)^{1/2} F_{\beta+\gamma }(u  )^{1/2} G_{\alpha-\beta +\kappa } (v)  \right.\\
&\hspace{2cm}\left. +\left( \sum_{\substack{ |\beta |> |\alpha |/2 \\ |\beta |\leq r}} +\tau \sum_{\substack{ |\beta |> |\alpha |/2 \\ |\beta |\geq r+1}}  \right) G_\beta (u) F_{\alpha -\beta +\kappa } (v)^{1/2} F_{\alpha - \beta + \kappa +\gamma }(v)^{1/2} \right)\\
&\lec\| v \|_{\widetilde{Y}(\tau )} \sum_{|\gamma |=2} \sum_{|\beta |>0} F_{\beta }(u)^{1/2} F_{\beta+\gamma }(u  )^{1/2}  \\
&\hspace{1cm}+\left( \sum_{0< |\beta |\leq r} +\tau \sum_{ |\beta |\geq r+1}  \right) G_\beta (u) \sum_{\substack{|\gamma | =2 \\ |\kappa | =1}}  \sum_{|\alpha |\geq r }F_{\alpha -\beta +\kappa } (v)^{1/2} F_{\alpha - \beta + \kappa +\gamma }(v)^{1/2} \\
&\lec \| v \|_{\widetilde{Y}(\tau )} \| u \|_{\widetilde{X}(\tau )} +\| v \|_{\widetilde{X}(\tau )}\left( \sum_{0< |\beta |\leq r} +\tau \sum_{ |\beta |\geq r+1}  \right) G_\beta (u) \\
&\lec\| v \|_{\widetilde{Y}(\tau )} \| u \|_{\widetilde{X}(\tau )} +\| v \|_{\widetilde{X}(\tau )}\| u \|_{\widetilde{Y}(\tau )},
\end{split}
\eqne
which proves the lemma.

It remains to show \eqref{bl} and~\eqref{bh}.
First, observe that
  \begin{equation}
   b_{\low } \leq\frac{\tau^{|\alpha |-r}}{\tau^{((|\beta |-r)\vee 0)/2}\tau^{((|\beta|+2-r)\vee 0)/2} \tau^{(|\alpha - \beta + \kappa |-r-1)\vee 0} }
   ,
   \label{EQ12}
  \end{equation}
using $\tau\leq 1$.
Note that we have $|\beta |\leq |\alpha |/2$, and so if $|\alpha - \beta |\leq r-1$, then 
   \label{EQ36}\begin{equation}b_{\low } \leq \tau^{|\alpha |-r
   -(|\beta |-r )/2 - (|\beta +\gamma |-r )/2}= \tau^{|\alpha |-|\beta
   |-1} \leq \tau^{|\alpha |/2 -1 } \leq 1.
     \end{equation}
 Otherwise, for $|\alpha - \beta |\geq r$, we have three cases:

\noindent\emph{Case 1.} $|\beta |\geq r$. Then
\begin{equation}
b_{\low } \leq \tau^{|\alpha | - r -\frac12 (|\beta | -r )-\frac12 (|\beta |+2-r ) -( |\alpha |-|\beta | - r )}=\tau^{r-1}\leq \tau .
   \label{EQ37}
     \end{equation}
\noindent\emph{Case 2.} $|\beta |\in \{ r-2,r-1 \}$. Then 
\begin{equation}
b_{\low } \leq \tau^{ |\alpha | - r -\frac12 (|\beta |+2-r ) -( |\alpha |-|\beta | - r  ) }= \tau^{ |\beta |/2 + r/2 -1  }\leq \tau^{  r-2 }\leq \tau  .
   \label{EQ38}
     \end{equation}
\noindent\emph{Case 3.} $|\beta |\leq r-3$. Then
\begin{equation}
b_{\low } \leq \tau^{|\alpha | - r  -( |\alpha |-|\beta | - r )}=\tau^{|\beta |} \leq \tau , 
   \label{EQ39}
     \end{equation}
as required.

As for \eqref{bh},
first note that
  \begin{equation}
   b_{\high } \leq \frac{\tau^{|\alpha |-r}}{\tau^{(|\beta |-r-1)\vee 0}\tau^{((|\alpha - \beta+\kappa  |-r)\vee 0)/2} \tau^{((|\alpha - \beta + \kappa |+2-r)\vee 0)/2} }
   ,
   \label{EQ13}
  \end{equation}
by $|\gamma|\leq2$.
Since
$|\beta |> |\alpha |/2$, and so, for $|\beta |\leq r$, 
\begin{equation}
b_{\high } \leq \tau^{|\alpha |-r -(|\alpha -\beta +\kappa |-r)/2-(|\alpha - \beta +\kappa |+2-r )/2}= \tau^{|\beta |-2 } \leq\tau^{|\alpha |/2 -3/2 } \leq 1.
   \label{EQ40}
     \end{equation}

Otherwise, for $|\beta |\geq r+1$, we have three cases:

\noindent\emph{Case 1.}  $|\alpha -\beta |\geq r-1$. Then
\begin{equation}
b_{\high }= \tau^{|\alpha | - r - (|\beta | - r-1) -\frac12 (|\alpha - \beta | -r+1 )-\frac12 (|\alpha - \beta  |-r+3 )}=\tau^{r-1}\leq \tau . 
   \label{EQ41}
     \end{equation}
\noindent\emph{Case 2.} $|\alpha -\beta |\in \{  r-3, r-2 \}$. Then
\begin{equation}
b_{\high } = \tau^{|\alpha | - r - (|\beta | - r-1) -\frac12 (|\alpha - \beta  |-r+3 ) }=\tau^{(|\alpha - \beta | + r -1)/2 }\leq  \tau^{r-2} \leq \tau . 
   \label{EQ43}
     \end{equation}
\noindent\emph{Case 3.} $|\alpha -\beta |\leq r-4 $. Then 
\begin{equation}
b_{\high } = \tau^{|\alpha | - r - (|\beta | - r-1)} =\tau^{|\alpha - \beta | +1}\leq \tau ,
   \label{EQ44}
     \end{equation}
as required.
\eold
\end{proof}

We note that the high and low factorial coefficients $a_{\rm h}$, $a_{\rm l}$, as well as the coefficients $b_{\rm h}$, $b_{\rm l}$ are not related to paraproduct estimates, but instead are a consequence of the product rule applied to \eqref{def_Salpha}, where we used  H\"older's inequality $\| fg\|\leq \|f \|_\infty \|g \|$ in \eqref{EQ30} differently for high and low $|\beta |\in (0,|\alpha |/2)$.\colb

\begin{remark}[The case of any dimension $d\geq 2$]
{\rm
It is easy to check that the above estimates can be performed for any $d\geq 2$.
Indeed, in a space dimension~$d$, we need to replace $D^2$ with $D^d$ in the embedding~\eqref{interp1}, so that $|\gamma |=d$ (instead of $|\gamma |=2$). A short calculation shows that the factorial estimates \eqref{calc_al}--\eqref{calc_ah} remain the same. As for the powers of $\tau$, we have $b_l =\tau^{E_l}$, $b_h =\tau^{E_h}$, where
\begin{equation}
\begin{split}
E_l &\coloneqq |\alpha | -d-1 -\frac12 \left( (|\beta | - d-1 )\vee 0 \right) - \frac12 \left( (|\beta | -1 )\vee 0 \right) -  \left( (|\alpha | - |\beta | - d-1 )\vee 0 \right) ,\\
E_h &\coloneqq |\alpha | -d-1 -\left( (|\beta | - d-2 )\vee 0 \right) - \frac12 \left( (|\alpha |- |\beta | -d )\vee 0 \right) - \frac12 \left( (|\alpha | - |\beta | )\vee 0 \right)
,
\end{split}
   \label{EQ45}
     \end{equation}
recalling that $r=d+1$. The main observation is that $E_l , E_h \geq 1$ as long as $|\alpha |\geq 2d+4$. Indeed, in the low case, i.e., when $|\beta |\leq |\alpha |/2$, we have $|\alpha - \beta | \geq |\alpha | - |\beta | \geq |\alpha |/2 \geq d+1$, Cases~1--3 above \eqref{EQ13} can be easily generalized as follows:\\
\emph{Case 1:} $|\beta |\geq d+1$. Then $E_l = |\alpha | -d-1 -\frac12 \left( |\beta | - d-1  \right) - \frac12 \left(|\beta | -1  \right) -  \left( |\alpha | - |\beta | - d-1 \right)=\frac{d+3}2 \geq 1$;\\
\emph{Case 2:} $1\leq |\beta |\leq d$. Then $E_l = |\alpha | -d-1  - \frac12 \left(|\beta | -1  \right) -  \left( |\alpha | - |\beta | - d-1 \right)=\frac{|\beta |+3}2 \geq 1$. \\
As for the high case $|\beta |> |\alpha |/2$, we have $|\beta |\geq d+2$, and:\\
\emph{Case 1:} $|\alpha - \beta |\geq n$. Then $E_h = |\alpha | -d-1 - \left( |\beta | - d-2  \right) - \frac12 \left(|\alpha |- |\beta | -d  \right) -  \frac12 \left( |\alpha | - |\beta |  \right)=\frac{d+2}2 \geq 1$;\\
\emph{Case 2:} $|\alpha - \beta |\leq d-1$. Then $E_h = |\alpha | -d-1 - \left( |\beta | - d-2  \right)  -  \frac12 \left( |\alpha | - |\beta |  \right)=\frac{d+2}2 \geq 1$. \\
If $a\coloneqq |\alpha  |\leq 2d+3 $, we simply observe that
\begin{equation}
\tau^{a-r} \| u\cdot \nabla v \|_{H^a} \lec \tau^{a-r} \| u \|_{H^r} \| v \|_{H^{a+1}} + \tau^{a-r} \| u \|_{H^a} \| v \|_{H^{r}} \lec  \| u \|_{H^r} \| v \|_{\tY} + \| u \|_{\tX } \| v \|_{H^r}
   \label{EQ46}
     \end{equation}
for any $a\geq r$, and the case $a\leq r-1$ follows similarly.

}
\end{remark}


\begin{remark}
{
\rm
We note that the sum in $S_\alpha (u,v) $
does not include the terms with $\beta=0$, which would result in a derivative of order $r+1$ in $\p^\beta u \cdot \nabla \p^{\alpha - \beta } v$.
This, together with our choice \eqref{EQ19} of the definitions of the analytic spaces $X(\tau )$ and $Y(\tau )$, ensures that the coefficients $b_l$, $b_h$ of the powers of $\tau$ are under sufficient control to yield the desired product estimate. In particular, thanks to this, the claim of Lemma~\ref{L_product} involves norms $\tX$ and $\tY$ (instead of, for example, $\oY$), providing a factor of $\tau$ in front of the $Y$ norm. This is essential for obtaining the persistence of analyticity; recall the comments following Theorems~\ref{T01} and~\ref{T02}.

On the other hand, note that the step showing that $a_l,a_h \lec_r 1$ above remains valid also in the case~$\beta =0$. In fact, in the case $\beta =0$ the same proof would give $\| \cdot \|_{Y(\tau )} +\| \cdot \|_{H^r}$ instead of $\| \cdot \|_{\widetilde{Y} (\tau )}$. We can thus obtain  the following.
}
\end{remark}


\cole
\begin{corollary}[Related product estimates]\label{cor_products}
For all $i,j\in \{ 1,\ldots , d\}$,
\eqnb\label{010}
\sum_{\substack{|\alpha |\geq r \\ \alpha_2\geq 2}} \frac{|\alpha |^r}{|\alpha |!} \tau^{|\alpha |-r} \epsilon^{\alpha_2} \| \p^{(\alpha_1,\alpha_2-2)} \nabla (\p_j u \p_i v)  \|_{L^2} \lec_r \| u \|_{\widetilde{X}(\tau )} \| v \|_{\widetilde{X}(\tau )} ,
\eqne
\eqnb\label{011}
\sum_{|\alpha |\geq r } \frac{|\alpha |^r}{|\alpha |!} \tau^{|\alpha |-r} \epsilon^{\alpha_2} \| \p_1^{|\alpha|-1} (\p_j u \p_i v)  \|_{L^2} \lec_r \| u \|_{\widetilde{X}(\tau )} \| v \|_{\widetilde{X}(\tau )} ,
\eqne
\eqnb\label{012}
\sum_{|\alpha |\geq r } \frac{|\alpha |^r}{|\alpha |!} \tau^{|\alpha |-r} \epsilon^{\alpha_2} \| \p_1^{|\alpha|-1} ( u \cdot \nabla v)  \|_{H^1} \lec_r \| u \|_{\tX (\tau )} \| v \|_{\oY(\tau )} .
\eqne
Similarly,
\eqnb\label{013}
 \|  u \cdot \nabla v \|_{X(\tau )}
  \lec_r \| u \|_{\tX (\tau )}  \| v \|_{{\oY}(\tau )}
  .
\eqne
\end{corollary}
\colb

\begin{proof}(Sketch) The proof follows in the same way as that of Lemma~\ref{L_product}, and we only check the nonnegativity of the powers of~$\tau$.
For \eqref{010}, we set $\wa \coloneqq (\alpha_1 ,\alpha_2 -2 ) + e_k$, where $k\in \{ 1,2\}$. We have $|\wa |= |\alpha |-1$. For $\beta \leq  \wa$, we estimate
\eqnb\label{aab}
\| \p^\beta \p_j u \p^{\wa - \beta } \p_i v \|_{L^2} \leq \| \p^\beta \p_j u \|_{L^\infty } \| \p^{\wa - \beta } \p_i v \|_{L^2}  \lec \| \p^\beta \p_j u \|_{L^2 }^{1/2}\| D^d \p^\beta \p_j u \|_{L^2 }^{1/2}  \| \p^{\wa - \beta } \p_i v \|_{L^2} 
\eqne
in the low case (i.e.,  $|\beta |\leq\frac{|\wa |}2$),  and
\eqnb\label{aac}
\| \p^\beta \p_j u \p^{\wa - \beta } \p_i v \|_{L^2} \leq \| \p^\beta \p_j u \|_{L^2 } \| \p^{\wa - \beta } \p_i v \|_{L^\infty}  \lec \| \p^\beta \p_j u \|_{L^2 } \| D^d \p^{\wa - \beta} \p_i v \|_{L^2 }^{1/2}  \| \p^{\wa - \beta } \p_i v \|_{L^2}^{1/2} 
\eqne
in the high case (i.e.,  $|\beta |>  \frac{|\wa |}2$). Thus, aiming for two $\tX$ norms on the right-hand side, we see that $b_l = \tau^{E_l}$, $b_h=\tau^{E_h}$, where
\eqnb\label{aad}
\begin{split}
E_l &\coloneqq |\alpha | -r - \frac12 \left( (|\beta | + 1-r ) \vee 0 \right) - \frac12 \left( ( |\beta |+1+n -r ) \vee 0 \right) - \left( |\wa |-|\beta |+1 -r \right) \vee 0,\\
E_h &\coloneqq |\alpha | -r -  \left( (|\beta | + 1-r ) \vee 0 \right) - \frac12 \left( ( |\wa | -|\beta |+1 -r ) \vee 0 \right) - \frac12 \left( (|\wa |-|\beta |+1+n -r) \vee 0 \right).
\end{split}
\eqne
Now, we verify that $E_l,E_h\geq 0$ as long as $|\alpha |\geq 2d+1$. Recall from Theorem~\ref{T01} that  $r=d+1$. Thus, in the low case, $|\wa - \beta | \geq | \wa | - | \beta |  \geq |\wa | - \frac{|\wa |}2 = \frac{|\alpha |-1}{2} \geq d$. Therefore:\\
\emph{Case 1, $|\beta | \geq d $}\/: Then $E_l=|\alpha |-d-1- \frac12 (|\beta |-d ) - \frac12 |\beta | - (|\wa |- |\beta | -d) = \frac{d}2 \geq 0$;\\
\emph{Case 2, $|\beta | \leq d-1 $}\/: Then $E_l=|\alpha |-d-1 - \frac12 |\beta | - (|\wa |- |\beta | -d) = \frac{|\beta |}2 \geq 0$.

On the other hand, in the high case, we have $|\beta | > |\wa |/2 \geq d$, and thus\\
\emph{Case 1, $|\wa - \beta | \geq d $}\/: Then $E_h=|\alpha |-d-1-  (|\beta |-d ) - \frac12 \left( |\wa | -|\beta |-d \right) - \frac12 (|\wa |- |\beta |) = \frac{d}2 \geq 0$,\\
\emph{Case 2, $|\wa - \beta | \leq d $}\/: Then $E_h=|\alpha |-d-1-  (|\beta |-d )  - \frac12 (|\wa |- |\beta |) = \frac{|\wa | - |\beta |}2 \geq 0$, as required. For $a\coloneqq |\alpha |\leq 2d $ we observe that
  \begin{equation}
   \tau^{a-r} \| \p_j u \p_i v \|_{H^{a-1}} \lec \tau^{a-r} \| u \|_{H^r} \| v \|_{H^{a}} + \tau^{a-r} \| u \|_{H^a} \| v \|_{H^{r}} \lec  \| u \|_{\tX} \| v \|_{\tX} 
   \llabel{EQ14}
   \end{equation}
for any $a\geq r$ (and similarly for $a<r$). 

The inequality~\eqref{011} is similar to \eqref{010}; in particular, the powers of $\tau$ are the same as above. 
For \eqref{012}, we only estimate the $\dot H^1$ part, in which case we set $\wa \coloneqq (|\alpha |-1, 0 ) + e_k$, where $k\in \{1,2\}$. We proceed analogously to \eqref{aab}--\eqref{aac}, obtaining
\begin{equation}
\begin{split}
E_l &\coloneqq |\alpha | -r - \frac12 \left( (|\beta | -r ) \vee 0 \right) - \frac12 \left( ( |\beta |+d -r ) \vee 0 \right) - \left( |\wa |-|\beta |+1 -(r+1) \right) \vee 0,\\
E_h &\coloneqq |\alpha | -r -  \left( (|\beta | -r ) \vee 0 \right) - \frac12 \left( ( |\wa | -|\beta |+1 -(r+1) ) \vee 0 \right) - \frac12 \left( (|\wa |-|\beta |+1+d -(r+1)) \vee 0 \right),
\end{split}
   \label{EQ47}
     \end{equation}
where we used the fact that, compared to~\eqref{aad}, we now do not have $\p_i$ next to $u$, and we aim to obtain $\| v \|_{\oY}$ (rather than $\| v \|_{\tX }$), so that some terms have $(r+1)$ instead of~$r$. Next, we verify that $E_l,E_h \geq 0$ if $|\alpha |\geq 2d+2$. Indeed, in the low case, $|\alpha | - |\beta |\geq |\alpha |/2 \geq d+1$, and so:\\
\emph{Case 1, $|\beta |\geq d+1$}\/: Then $E_l = |\alpha | -r - \frac12 |\beta | + \frac{r}{2} - \frac12 |\beta | + \frac12 - |\alpha |+ |\beta |+ (d+1)= 1+ \frac{d}2 \geq 0$,\\
\emph{Case 2,  $|\beta |\leq d$}\/: Then $E_l = |\alpha | -r  - \frac12 |\beta | + \frac12 - |\alpha |+ |\beta |+ (d+1)=  \frac{|\beta |+1}2 \geq 0$.\\
In the high case, we have $ |\beta |> |\alpha |/2 \geq d+1$, and so:\\
\emph{Case 1, $|\alpha | -|\beta |\geq d+1$}\/: Then $E_h = |\alpha | -r -  |\beta | + {r} - \frac12 \left( |\alpha |- |\beta |- (d+1) \right)  - \frac12 \left( |\alpha |- |\beta |-1 \right) =  \frac{d+2}2 \geq 0$,\\
\emph{Case 2, $1\leq |\alpha | -|\beta |\leq d$}\/: Then $E_h = |\alpha | -r -  |\beta | + {r}  - \frac12 \left( |\alpha |- |\beta |-1 \right) =  \frac{|\alpha | - |\beta | +1 }2 \geq 0$,\\
\emph{Case 3, $ |\alpha | -|\beta |=0 $}\/: Then $E_h =0$, as required.
For $a\coloneqq |\alpha |\leq 2d+1 $, we observe that
\begin{equation}
\tau^{a-r} \| u \cdot \nabla v \|_{H^{a}} \lec \tau^{a-r} \| u \|_{H^r} \| v \|_{H^{a+1}} + \tau^{a-r} \| u \|_{H^a} \| v \|_{H^{r}} \lec  \| u \|_{\tX} \| v \|_{\oY}
,
   \llabel{EQ15}
     \end{equation}
for any $a\geq r$ (and similarly for $a<r$). The case of \eqref{013} is analogous to \eqref{012} addressed above.
\end{proof}
\colb

For future reference, we state the following general product
estimate for the $\tX$ analytic norm.

\begin{lemma}(product estimate in the flat case)\label{L01}
{With $r= d+1$}, we have
\begin{equation}
\Vert u v\Vert_{\tilde X(\tau)}
\lec_{r}
 \| u \|_{\widetilde{X}(\tau )} \| v \|_{\widetilde{X}(\tau )}
 ,
   \label{EQ16}
     \end{equation}
for all $\tau\in(0,1]$ and all scalar functions~$u$ and~$v$.
\end{lemma}

The proof follows the same steps as that of Corollary~\ref{cor_products}.
We emphasize that $r=d+1$ can be relaxed to $r>d/2$, using the same
method, as the reader can check.
Observe that the lemma can be immediately generalized to an arbitrary
number of factors, and then, using the Taylor series, one can deduce a
chain rule.

We can now estimate the pressure $p$, which satisfies the elliptic problem 
\begin{align}
  \begin{split}
   &\Delta p
     =
       -\nabla (v+\ou ) : \nabla (v+\ou ) 
    \inon{in $\Omega$}
    \\&
    \partial_{\mathsf{n}}p
    =
   - \ou \cdot  \nabla (\ou_2 + v_2) 
    - v_1 \p_1 \ou_2-\p_t \ou_2 \inon{on $\partial\Omega$}
    ,
    \end{split}
   \label{EQ69}
  \end{align}
where we set
  \begin{equation}
   A : B
   =
   A_{ij}  B_{ji}
   .
   \label{EQ07}
  \end{equation}
\colb  

The next lemma provides a general result on how to estimate the analytic norm of~$\na p$.

\cole
\begin{lemma}[The pressure estimate in the analytic norm]
\label{L02}
Consider the solution $p$ of the Neumann problem
  \begin{align}
  \begin{split}
   &\Delta p = \nabla f_1 : \nabla f_2\hspace{0.45cm}
   \inon{in $\Omega$},
   \\&
   \partial_{2} p = (g_1\cdot \nabla g_2)|_{\partial\Omega} 
   \inon{on $\partial\Omega$}
   ,
  \end{split}
   \label{EQ25}
  \end{align}
with the normalizing condition $\int_\Omega p= 0$,
where $g$ is defined in $\Omega$,
with the compatibility condition
  \begin{equation}
\int \nabla f_1 : \nabla f_2
=
\int_{\partial\Omega}
  (g_1\cdot \nabla g_2)
       \cdot \mathsf{n}
   .
   \label{EQ115}
  \end{equation}
Then we have
  \begin{align}
  \begin{split}
   &
       \Vert \nabla p\Vert_{X(\tau )}
   \lec
    \Vert f_1\Vert_{\tX(\tau)}
    \Vert f_2\Vert_{\tX(\tau)}
    +  \Vert g_1 \Vert_{\tX(\tau)}\Vert g_2 \Vert_{\oY(\tau)}
  ,
  \end{split}
   \label{EQ27}
  \end{align}
provided $\epsilon >0$ is sufficiently small.
\end{lemma}
\colb

Note that the first equation in \eqref{EQ25}
does not loose
a derivative (so that we have two $\tX$ norms in \eqref{EQ27}), while the second equation in \eqref{EQ25}
does,
so that one of the norms becomes~$\oY$.

\begin{proof}[Proof of Lemma~\ref{L02}]
Denote the left-hand side of \eqref{EQ27} by~$I$.
Suppose that $\alpha_2\geq 2$. Then we write
\eqnb\label{transfer}
\p^\alpha \p_j p = (\Delta - \p_1^2 ) \p_1^{\alpha_1} \p_2^{\alpha_2-2} \p_j p = \p_1^{\alpha_1} \p_2^{\alpha_2-2}\p_j (\nabla f_1  \nabla f_2)- \p^{\overline{\alpha }} \p_j p
,
\eqne
for $j=1,2$, where 
\begin{equation}
\overline{\alpha } \coloneqq (\alpha_1+2 , \alpha_2-2 )
;
   \label{EQ49}
     \end{equation}
we also used a symbolic notation, omitting ``$:$'' and non-significant indices.
By \eqref{010}, we have
\eqnb\label{006}
\sum_{j=1,2}\sum_{\substack{|\alpha |\geq r \\ \alpha_2\geq 2}} \frac{|\alpha |^r }{|\alpha |! } \tau^{|\alpha |-r} \epsilon^{\alpha_2}\left(  \| \p^{(\alpha_1 ,\alpha_2-2) } \p_j (\na f_1 \na f_2 )\| + \| \p^{\overline{\alpha}} \na p\| \right) \lec \| f_1 \|_{\tX(\tau )} \| f_2 \|_{\tX(\tau )} +  \epsilon^2 \| \na p \|_{X(\tau )},
\eqne
where we used $|\overline{\alpha }|=|\alpha |$ in the last inequality.

If $\alpha_2 = 1$, we apply the $H^2$ estimate to \eqref{EQ69} to obtain
  \begin{equation}
   \| \p^\alpha \p_j p \|\leq \| \p_1^{\alpha_1} p \|_{H^2} \leq \| \p_1^{\alpha_1} (\na f_1 \na f_2 ) \| + \| \p_1^{\alpha_1} (g_1\cdot \nabla g_2) \|_{H^1 }
   \llabel{EQ10}
   \end{equation}
and similarly $\| \p^\alpha \p_j p \| \leq  \| \p_1^{\alpha_1-1} (\na f_1 \na f_2 ) \| + \| \p_1^{\alpha_1-1} (g_1\cdot \nabla g_2) \|_{H^1 }$ if~$\alpha_2=0$.
Using \eqref{010}--\eqref{012}, we thus derive
\begin{equation}\begin{split}
\sum_{j=1,2}\sum_{\substack{|\alpha |\geq r \\ \alpha_2\leq 1}} \frac{|\alpha |^r }{|\alpha |! } \tau^{|\alpha |-r} \epsilon^{\alpha_2} \| \p^\alpha \p_j p \| &\lec \sum_{j=1,2}\sum_{\substack{|\alpha |\geq r \\ \alpha_2\leq 1}} \frac{|\alpha |^r }{|\alpha |! } \tau^{|\alpha |-r} \epsilon^{\alpha_2}
\left(  \| \p_1^{|\alpha |-1}  (\na f_1 \na f_2 )\| + \| \p_1^{|\alpha_1 |-1 } (g_1\cdot \nabla g_2 )\|_{H^1} \right)\\
&\lec \| f_1 \|_{\tX(\tau )} \| f_2 \|_{\tX(\tau )} + \| g_1 \|_{\tX(\tau )} \| g_2 \|_{\oY(\tau )}.
\end{split}
   \label{EQ50}
     \end{equation}
This and \eqref{006} give
\begin{equation}
\| \na p \|_{X(\tau )} \lec   \| f_1 \|_{\tX(\tau )} \| f_2 \|_{\tX(\tau )} + \| g_1 \|_{\tX(\tau )} \| g_2 \|_{\oY(\tau )} + \epsilon^2 \| \na p \|_{X(\tau )} ,
   \llabel{EQ09}
     \end{equation}
and absorbing the last term by the left-hand side gives the claim.
\end{proof}

As for the term $\p_t \ou$ in the equation~\eqref{EQ69} for $p$, we can adjust the above proof to the case of simpler boundary conditions. 

\begin{corollary}\label{cor_extra_p_flat}
Suppose that $p$ is a solution to
\eqnb
\begin{split}
\Delta p &=0 \quad \text{ in }\Omega,\\
\p_2 p &=f \quad \text{ on }\p \Omega,
\end{split}
\eqne
where
$\int_{\p \Omega } f =0$.
Then
\begin{equation}
\| \na p \|_{\tX } \lec \| f \|_{\tX }
.
   \label{EQ51}
     \end{equation}
\end{corollary}

\begin{proof}[Proof of Corollary~\ref{cor_extra_p_flat}]
Similarly as in the above proof, we have that if $\alpha_2 \geq 2$, then $\p^\alpha \p_j p =\p^{\overline{\alpha }} \p_j p$ for $j=1,2$, where $\overline\alpha$ is as in~\eqref{EQ49}. If $\alpha_2=1$, then the $H^2$~estimate gives
\begin{equation}
\| \p^\alpha \p_j p \| \leq \| \p^{\alpha_1}_1 p \|_{H^2} \leq \| \p^{\alpha_1}_1 f \|_{H^{1/2} (\Omega )} \lec \| \p^\alpha f \| + \| \p_1^{\alpha_1+1} f \| + \| \p_1^{\alpha_1} f \| ,
   \label{EQ52}
     \end{equation}
and similarly, if $\alpha_2=0$ and $|\alpha |\geq r$, then
\begin{equation}
\| \p^\alpha \p_j p \|  \lec \| \p^{\alpha_1-1}_1 \p_2 f \| + \| \p_1^{\alpha_1} f \| + \| \p_1^{\alpha_1-1} f \|
.
   \label{EQ53}
     \end{equation}
Therefore,
\begin{equation}
\sum_{j=1,2}\sum_{|\alpha |\geq r } \frac{|\alpha |^r }{|\alpha |! } \tau^{|\alpha |-r} \epsilon^{\alpha_2} \| \p^\alpha \p_j p \| \lec \epsilon^2 \| \nabla p \|_{\tX} + \| f \|_{\tX},
   \label{EQ54}
     \end{equation}
since $\epsilon \leq 1$, $\tau \leq \tau_0 \lec 1$, and $\frac{|\alpha |^r }{|\alpha |!} \lec_r \frac{(|\alpha |-1)^r}{(|\alpha |-1 )!}$, and the claim follows.
\end{proof}

\colb
Recalling the boundary value problem \eqref{EQ69} for $p$, the pressure estimates from Lemma~\ref{L02}, and Corollary~\ref{cor_extra_p_flat} yield
\eqnb\label{009}
\| \nabla p \|_{\tX }  \lec_r \| v+\ou \|_{\tX } \| v+\ou \|_{\oY }  + \| \p_t \ou \|_{\tX}
.
\eqne
This inequality and the product estimate from Lemma~\ref{L_product} allow us to sum the bounds \eqref{EQ08} on the time evolution of $\|\p^\alpha v \|$ to obtain
\begin{equation}\begin{split}
\frac{\d }{\d t } \| v \|_{X} &- \dot \tau \| v \|_{Y} = \sum_{|\alpha |\geq r} \frac{|\alpha |^r }{|\alpha |!} \tau^{|\alpha |-r} \epsilon^{\alpha_2}\frac{\d }{\d t }  \| \p^\alpha v \| \\
& \lec_r  \| v+\ou  \|_{\widetilde{X}} \| v+\ou \|_{\tY }+\| \ou  \|_{\widetilde{X}} \| v+\ou \|_{\oY }
\\&
+\| \ou  \|_{\oY } \| v+\ou \|_{\tX }+  \| \ou \|_\infty  (\| v \|_{\widetilde{X} } + \| \nabla v \|_{\widetilde{X}} ) +  \| v +\ou\|_{H^r} \| \nabla \ou \|_{\widetilde{X}} + \| \p_t \ou \|_{\tX }  \\
&\lec_r \| v \|_{\widetilde{X}} \| v \|_{\tY} + \| \ou \|_{\overline{Y}} \| v \|_{\widetilde{X}} + \| \ou \|_{\widetilde{X}} \| v \|_{\overline{Y}}+ \| \ou \|_{\widetilde{X}} \| \ou \|_{\overline{Y}}  + \| \p_t \ou \|_{\tX } ,
   \label{EQ55}\end{split}
     \end{equation}
proving the desired a~priori estimate \eqref{002}, where we also used \eqref{grad_in_X_is_Y} in the last line.

\subsection{Proof of Theorem~\ref{T01}}\label{sec_flat_conclude}
We can now use the a~priori estimate \eqref{002} to complete the proof of Theorem~\ref{T01}.
Recall that  $\tau \colon [0,T_0] \to [0,1]$ is given by \eqref{def_tau}, so that $\dot \tau = - M$, where $M>0$.
We show that
for every $v_0 \in X(\tau(0))\cap Y(\tau (0))$,
there exist sufficiently large $M>1$, sufficiently small $T_0>0$, and a unique local-in-time solution $v \in C([0,T_0]; \widetilde{X}(\tau ))$
to \eqref{intro_euler_v} satisfying
\eqnb\label{bound_on_construction}
\sup_{t\in [0,T_0 ]} \| v \|_{\widetilde{X}} + M \int_0^{T_0}  \| v \|_{\overline{Y}} \leq A,
\eqne
where 
\eqnb\label{def_A}
A\coloneqq 4  \| v_0 \|_{\tX (\tau (0))} +  \sup_{t\in [0,1]} \| \ou \|_{\overline{Y}(\tau )} +2 \int_0^t \| \p_t \ou \|_{\tX}.
\eqne
We consider a sequence $v^{(n)}$ defined by 
\begin{align}
  \begin{split}
   &
   \partial_{t} \vnp
     + (\vn+\ou )\cdot \nabla (\vn +\ou )
     + \nabla \pn+\p_t \ou  = 0
    ,
  \end{split}
   \label{EQ32}
  \end{align}
where
\begin{align}
  \begin{split}
   &\Delta \pn
     =
       -\nabla (\vn+\ou ) : \nabla (\vn +\ou ) 
    \inon{in $\Omega$}
    \\&
   \p_2  \pn
    =
    (
       \vn_i - \ou_i )\cdot \na (\vn_j - \ou_j )_2 -\p_t \ou_2
    \inon{on $\partial\Omega$}
    ,
    \end{split}
   \label{EQ34}
  \end{align}
with the initial iterate
\begin{align}
  \begin{split}
   &v^{(0)} \coloneqq v_0 
   .
  \end{split}
   \label{EQ33}
  \end{align}
We first observe that 
  \begin{equation}
   \nabla \cdot \vn=0
   \inon{in $\Omega$}
   \label{EQ42}
  \end{equation} 
and 
   \begin{equation}
   \vn_2=0   
   \inon{on $\partial\Omega$}
   ,
   \label{EQ71}
  \end{equation}
for all $n\in\mathbb{N}$,
thanks to~\eqref{EQ34}. In order to send $n\to \infty$, we first show that if $M\geq1$ is sufficiently large and $T_0>0$ is sufficiently small, then the a~priori bound \eqref{bound_on_construction} holds for all iterates, i.e.,
\eqnb\label{unif_bounds}
\sup_{t\in [0,T_0 ]} \| \vn \|_{\widetilde{X}} + M \int_0^{T_0}  \| \vn \|_{\overline{Y}} \leq A
,
\eqne
for all $n\geq 0$. 
As in \eqref{002}, we see that
\eqnb\label{014}
\frac{\d }{\d t} \| \vnp \|_{\widetilde{X}} +M \| \vnp \|_{\oY} \lec_r \left( \| \vn \|_{\tX}  + \| \ou \|_{\tX } \right) \left( \| \vn \|_{\oY}  + \| \ou \|_{\oY } \right)   + \| \ou \|_{\overline{Y}}^2 + \|\p_t \ou \|_{\tX},
\eqne
for every $n\geq 1$. Note that we have $\| \cdot \|_{\overline{Y}}$ instead of $\| \cdot \|_{\widetilde{Y}}$ (as in \eqref{002}), since there is no cancellation in the nonlinear terms of \eqref{EQ32} when tested against~$\vnp$; namely, we use \eqref{013} instead of the product estimate (Lemma~\ref{L_product}). It is an advantage of the analytic approach that this does not matter.

The claim \eqref{unif_bounds} is trivial for $n=0$, and, for $n\geq 1$, \eqref{014} gives
\begin{equation}\begin{split}
\sup_{t\in [0,T_0 ]} \| \vnp \|_{\widetilde{X}} &+ M \int_0^{T_0}  \| \vnp \|_{\oY} \leq \left( 1+ {T_0}M \right) \sup_{t\in [0,T_0 ]} \| \vnp \|_{\widetilde{X}} + M \int_0^{T_0}  \| \vnp \|_{Y} \\
&\leq 2 \| v_0 \|_{X(\tau (0))} +C_0\sup_{t\in [0,T_0]}  \left( \| \vn  \|_{\widetilde{X}}+\|  \ou \|_{\widetilde{X}}\right) \,\,\ \int_0^{T_0} \left( \| \vn \|_{\overline{Y}} +  \| \ou  \|_{\overline{Y}} \right) + \int_0^{T_0} \| \p_t \ou \|_{\tX } \\
&\leq 2 \| v_0 \|_{X(\tau (0))}  + 2 C_0 A^2 (M^{-1} + T_0) + \int_0^{T_0} \| \p_t \ou \|_{\tX } \leq A,
   \label{EQ56}\end{split}
     \end{equation}
provided
\eqnb\label{choice_M_T0}
M \geq 12C_0 A \qquad \text{ and }\qquad T_0 \leq \min \{1,\tau_0 \} /M.
\eqne

Note that, due to the linear choice \eqref{def_tau} of $\tau$, the choice of $T_0$ also implies that $\tau \geq 0$ for $t\in [0,T_0]$, thus leading to~\eqref{unif_bounds}.

 Now, we provide some details for the fixed point argument.
Subtracting \eqref{EQ32} for $n$ and $n-1$ and setting
\begin{equation}
w^{(n+1)} \coloneqq \vnp -\vn
,
   \label{EQ57}
     \end{equation}
we get
  \begin{align}
  \begin{split}
   &
   \partial_{t} w^{(n+1)}
      + (\vn +\ou ) \cdot\nabla  w^{(n)} + w^{(n)} \cdot \nabla (\vnm +\ou )      + \nabla (\pn - \pnm) = 0
    .
  \end{split}
   \label{EQ48}
  \end{align}
As in \eqref{014}, we obtain
\begin{equation}\begin{split}
\frac{\d }{\d t} \| w^{(n+1)} \|_{\widetilde{X} } + M \| w^{(n+1)}\|_{\overline{Y} }&\lec \| w^{(n)}  \|_{\widetilde{X} } \left( \| \vn \|_{\overline{Y} } +\| \vnm  \|_{\overline{Y} }  \right)+ \left( \| \vn \|_{\widetilde{X} }+\| \vnm \|_{\widetilde{X} } \right)  \| w^{(n)} \|_{\overline{Y} }  \\
&\quad +\| w^{(n)} \|_{\widetilde{X} } \| \ou \|_{\overline{Y} } +\| \ou \|_{\widetilde{X} } \| w^{(n)} \|_{\overline{Y} }.
\end{split}    
   \label{EQ58}
     \end{equation}
Setting 
\begin{equation}
a_n \coloneqq \sup_{t\in [0,T_0]}\| w^{(n)} \|_{\widetilde{X} },\qquad b_n \coloneqq \int_0^{T_0} \| w^{(n)} \|_{\overline{Y} },
   \label{EQ59}
     \end{equation}
we may integrate on $[0,t]$, for $t\in [0,T_0]$ to obtain
\eqnb\label{015}\begin{split}
 a_{n+1}+ Mb_{n+1} &\lec a_{n} \frac{2A}{M}+ 2A b_{n} + a_{n} \int_0^{T_0} \| \ou \|_{\overline{Y}} + b_{n} \sup_{t\in [0,T_0]} \| \ou \|_{\widetilde{X}} \\
 &\lec (a_{n} + Mb_{n} ) \frac{4A }M,
 \end{split}
\eqne
where we also recalled \eqref{choice_M_T0} that
$T_0\leq 1/M$. This shows that, choosing $M$ sufficiently large, and then $T_0$ sufficiently small, we have
that 
\begin{equation}
\{ \vn \}  \text{ is Cauchy in } C^0 \left( [0,T_0]; X(\tau(t)) \right)
,
   \label{EQ60}
     \end{equation}
and so $\vn \stackrel{n\to \infty}{\longrightarrow} v$ in $C^0 \left( [0,T_0]; X(\tau(t)) \right)$ for some $v \in C^0 \left( [0,T_0]; X(\tau(t)) \right)\cap L^1 \left( (0,T_0); Y(\tau (t)) \right)$, where we also used Fatou's Lemma. The a~priori bound \eqref{bound_on_construction} follows by taking the limit $n\to\infty$ of \eqref{unif_bounds} and applying Fatou's Lemma, and, due to the strong mode of convergence, the impermeability conditions \eqref{EQ42} and the divergence-free conditions \eqref{EQ71}, the function  $v $ is a solution to~\eqref{intro_euler_v}.

Thanks to the a~priori bound \eqref{bound_on_construction}, the  uniqueness follows similarly to \eqref{015}
in a standard manner.

\section{The general case}
\label{sec3}
In this section, we prove Theorem~\ref{T01} in full generality. 

\subsection{Preliminaries}\label{sec_prelims}

Here, we define tools that are needed specifically for handling the general domain. We say that a domain $\Omega$ is analytic if it is locally a graph of an analytic function.

We introduce the Komatsu convention~\cite{K1} on derivatives,
\eqnb\label{komatsu_not}
\p^i \coloneqq \bigsqcup_{\alpha \in \{ 1,2\}^i} D^\alpha u,
\eqne
where 
\begin{equation}
D^\alpha u \coloneqq \p_{\alpha_1} \ldots  \p_{\alpha_i} u.
   \label{EQ63}
     \end{equation}
For example, $\p^2$ is a list of four second order derivatives.
Note that mixed derivatives are included in~$\p^i$ multiple times. 
As a consequence of this notation, 
\begin{equation}
\| D^i u \| = \sum_{\alpha \in \{ 1,2\}^i} \| D^\alpha u \|.
   \label{EQ64}
     \end{equation}
Moreover, we have the product rule 
\eqnb\label{komatsu_product}
\p^i (f g) - f \p^i g = \sum_{k=1}^i {i \choose k } (\p^k f ) (\p^{i-k} g ),
\eqne
where the last product denotes the tensor product of derivatives which is consistent with \eqref{komatsu_not}
(see (5.1.3) in Komatsu~\cite{K1}). The product rule \eqref{komatsu_product} is a consequence of the  identity
\begin{equation}
\sum_{\substack{\alpha' \subset \alpha \\ |\alpha' |=k }} {\alpha \choose \alpha' } = {m \choose k},
   \label{EQ66}
     \end{equation}
for each $m\in \N$, $k\in \{ 0 , \ldots , m \}$, and~$|\alpha |=m$.

We say that a vector field $T=\sum_{i=1}^n a_i \p_i$, where $a_i \colon \overline\Omega \to \R$ are analytic, is a \emph{tangential operator} to $\p \Omega$ if $T$ is a analytic vector field such that $T\delta =0$ on $\p \Omega$, where $\delta (x) $ is the distance function to the boundary $\p \Omega$, taking positive values inside $\Omega$ and negative outside; see \cite[Section~2.1]{CKV}
and \cite{JKL2}
for extensive discussions on tangential operators. We will abuse the notation, and denote by $T$  the tensor of a complete system of tangential derivatives, namely a system such that  the components of~$T$ span the tangent space at $x$, for each $x\in \p \Omega $; see  \cite[Theorem~3.1]{JKL2} for details. The system is chosen depending on the domain~$\Omega$. Recall from~\cite{JKL2} that we can use one tangential derivative in~2D and three in~3D, with a constant number $K\in \N$ in higher space dimensions.

We denote by 
\[
T^j \coloneqq \bigsqcup_{\beta \in \{ 1,\ldots , K\}^j} T^\beta u,
\]
the $j$-th tangential derivative, where we used Komatsu notation \eqref{komatsu_not} for the complete system $T=(T_1,\ldots , T_K)$ of tangential derivatives.

Note that we have
\eqnb\label{n_is_ana}
\mathsf{n} \in \tX ( \tau_0; \Omega')
\eqne
(see \cite{KP}), where $\Omega' \subset \Omega$ is some neighborhood of $\partial \Omega$ in $\Omega$, and $\| f\|_{\tX (\tau ; \Omega')}\coloneqq \|  f\|_{X (\tau ; \Omega')}+ \| f \|_{H^3(\Omega')}$,
\begin{equation}
\| f\|_{X(\tau;\Omega')} \coloneqq \sum_{i+j\geq r} \frac{(i+j)^r}{(i+j)!} \tau^{i+j-r}  \| \p^i T^j f \|_{L^2 (\Omega')}
   \label{EQ61}
     \end{equation}
denote the analytic spaces $\tX$, $X$ on the subdomain~$\Omega'$.

We also assume that 
\eqnb\label{D2phi_is_ana}
D^2 \phi \in \tX (\tau_0;\Omega')
,
\eqne
for some neighborhood $\Omega'$ of $\p\Omega$ in $\Omega$, where $\phi(x) \coloneqq \mathrm{dist}\,(x,\p \Omega)\colon \overline{\Omega} \to [0,\infty )$ denotes the signed distance function to~$\p \Omega$.

In the next statement, we recall the commutator estimates between $T^j$ and partial derivatives of order $1$ or~$2$.

\cole
\begin{lemma}[Commutator estimates]\label{L_com_curved}
We have that
\begin{equation}
[T^j, \p_k ] = \sum_{i=1}^d b_{j}^i \p_i
   \label{EQ62}
     \end{equation}
and
\eqnb\label{202}
[T^j , \Delta ] = \sum_{l=1}^j \sum_{|\alpha |=l} \frac{j!}{\alpha! (j-l)!} b_{\alpha_1} \cdot \nabla (b_{\alpha_2 } \cdot \nabla T^{j-l})   ,
\eqne
where
\eqnb\label{bounds_b}
\max_{|\alpha |=k } \| \p^\alpha b_j^i \|_\infty \lec (k+j)! K_1^k K_2^j
,
\eqne
for each $i=1,\ldots , d$, and $b_{\alpha_1},b_{\alpha_2}$ satisfy estimates analogous to~\eqref{bounds_b}.
\end{lemma}
\colb

\begin{proof}[Proof of Lemma~\ref{L_com_curved}]
See \cite[Lemma~3.5]{CKV} (or \cite[Section~5]{K1}) and \cite[p.~3385]{CKV}, respectively.
\end{proof}

Using the tangential operator, we can estimate the pressure function using the Komatsu notation: If $\Delta p = f$ in $\Omega$, then 
\[
\| p \|_{H^2} \lec \| f \| + \| p \|_{H^{3/2} (\p \Omega)} \lec \| f \| + \| T p \|_{H^{1/2} (\p \Omega)} + \|  p \|_{H^{1/2} (\p \Omega)} + \lec \| f \| + \| T p \|_{H^{1}} + \|  p \|_{H^{1}} 
\]
and so, interpolating $H^1$ between $H^2$ and $L^2$, 
\eqnb\label{H2_est}
\| p \|_{H^2} \lec \| f \| + \| T p \|_{H^1} + \| p \|.
\eqne
In particular, for every $i\geq 2,j\geq 0$, we may write $\Delta (\p^{i-2} T^j p) = \p^{i-2} T^j \Delta p  + \p^{i-2} [ \Delta , T^j] p$, so that \eqref{H2_est} gives
\eqnb\label{i2jany}
\| \p^i T^j p \| \lec \| \p^{i-2} T^j \Delta p \|_{H^2} + \| T \p^{i-2} T^{k} p \|_{H^1} + \| \p^{i-2} T^k p \| + \| \p^{i-2} [\Delta ,T^j ] p\|.
\eqne
Similarly, for $i=1$, $j\geq 1$ we have $\| \nabla T^j p \| \lec \| T^{j-1} p \|_{H^2}$, and so
\eqnb\label{i1j1}
\| \na T^j p \| \leq \| T^{j-1} \Delta p \| + \| T^j p \|_{H^1} + \| p \| + \| [\Delta , T^j ] p \|.
\eqne

\subsection{Proof of Theorem~\ref{T01}}\label{sec_curved}

Thanks to the Komatsu notation \eqref{komatsu_not}, the claim of Theorem~\ref{T01} follows in the same way as in the flat case (recall Section~\ref{sec_flat_conclude}) by replacing the analytic norms \eqref{EQ19} for the flat domain by the general analytic norms \eqref{EQ19c}, and by observing  the product estimate  
\eqnb\label{product_curved}
\sum_{i+j\geq r }c_{i,j}   \| S_{i j} (u,v)\|  \lec_r \| v \|_{\widetilde{Y}} \| u \|_{\widetilde{X}} +\| v \|_{\widetilde{X}} \| u \|_{\widetilde{Y}},
\eqne
where \begin{equation}
S_{i j} (v,w ) \coloneqq \p^i  T^j \left( (v\cdot \nabla )w 
   \label{EQ67}\right) - v\cdot \nabla \p^i T^j w,
     \end{equation}
and the pressure estimate
\eqnb\label{pres_gen}
\| \na p \|_{\widetilde{X}} \lec \| v \|_{\tX }\| v  \|_{\tX } + \| v \|_{\tX }\| \ou  \|_{\oY } +\| \ou  \|_{\tX }\| v  \|_{\oY } + \| \ou \|_{\tX }\| \ou  \|_{\oY } + \| \p_t \ou \|_{\tX},
\eqne
which we prove in Section~\ref{sec_prod_gen} and Section~\ref{sec_pressure_gen} (respectively) below. In particular, the a~priori estimate \eqref{002} follows by noting that $\| v \|_{\tX } \lec  \| v \|_{\tY}$.

\subsection{Product estimate: the general case}\label{sec_prod_gen}

Here we show~\eqref{product_curved}.  We start by using the product rule \eqref{komatsu_product} to  extract the commutators,
\begin{equation}
\begin{split}
\| S_{i j } (v,w ) \| &
\leq   \|  \p^i \p_k T^j (v_k w)  - v_k \p_k \p^i T^j w\|+ \| \p^i [T^j , \p_k ] (v_kw) \|\\
&\lec \sum_{(n , l) \ne 0} {i \choose n } {j \choose l}  \|   \p^n T^l v_k \p_k \p^{i - n } T^{j-l} w  \|+  \sum_{p=0}^i  {i\choose p} \| \p^{i-p} b_j  \p^{p+1} (v w)  \| \\
&=:  A+B,
\end{split}
   \label{EQ68}
     \end{equation}
where we write ``$ (n , l )\ne 0 $'' instead of ``$ n \in \{ 0, \ldots , i\}, l\in \{ 0,\ldots , j \}, (n , l) \ne 0$,'' for brevity.

For the main product term $A$, we have 
\begin{equation}
\begin{split}
\|   \p^n T^l v \p_k \p^{i-n } T^{j-l} w  \| &\leq  \| \p^n T^l v \|_{\infty } \| \p^{i-n +1 } T^{j-l} w \|\\
&\leq     \| \p^{n } T^l v \|  \| \p^{i - \beta +1 } T^{j-l} w \|  +\| \p^{n } T^l v \|^{1/2}  \| \p^{n +2 } T^l v \|^{1/2}  \| \p^{i - n +1 } T^{j-l} w \|  
,
\end{split}
   \label{EQ70}
     \end{equation}
for  $n+ l \leq (i +j )/2$, and similarly, 
\begin{equation}
\begin{split}
\|
&( \p^n T^l v \cdot   \nabla )\p^{i-n } T^{j-l} w  \| \leq \| \p^n T^l v \| \| \p^{i-n +1 } T^{j-l} w \|_{\infty } \\
&\indeq\leq    \| \p^{n } T^l v \|  \| \p^{i -n  +1 } T^{j-l} w \|  +\| \p^{n  } T^l v \|  \| \p^{i-n  +3} T^{j-l} w \|^{1/2} \| \p^{i-n +1  } T^{j-l} w \|^{1/2}    .
\end{split}
   \label{EQ72}
     \end{equation}
We shall use the notation
\eqnb\label{def_F}\begin{split}
F_{i,j} (u )& \coloneqq \begin{cases} \frac{\overline\epsilon^{i}\epsilon^{j} (i+j)^r \tau^{i+j -r }}{(i+j)!} \| \p^i T^j u \| \qquad &i+j\geq r ,\\
\| \p^i T^j u \| &i+j <r,
\end{cases}\\
 G_{i,j} (u) &\coloneqq \begin{cases} \frac{\overline\epsilon^{i}\epsilon^{j } (i+j)^{r+1} \tau^{i+j -r -1}}{(i+j)!} \| \p^i T^j u \| \qquad &i+j\geq r+1 ,\\
 \| \p^i T^j u \| &i+j\leq r,
 \end{cases}
\end{split}\eqne
so that $\sum_{i,j \geq 0 } F_{i,j} (u) \lec \| u \|_{\widetilde{X} (\tau )}$ and $\sum_{i,j \geq 0 } G_{i,j} (u) \lec \| u \|_{\widetilde{Y} (\tau )}$. We obtain
\begin{equation}
\begin{split}
\sum_{i+j \geq r} c_{i,j} A &\lec  \sum_{i+j\geq r }\biggl(  \sum_{0< n+l\leq (i+j)/2 } a_{\low } b_{\low } F_{nl}(u)^{1/2} F_{n+1,l}(u  )^{1/2} G_{(i-n+1)(j-l)} (v)
\\
&\hspace{3cm}
+\sum_{ n+l> (i+j)/2 } a_{\high } b_{\high } G_{nl} (u) F_{i-n+1,j-l} (v)^{1/2} F_{i-n+3,j-l }(v)^{1/2} \biggr),
\end{split}
   \label{EQ73}
     \end{equation}
where the terms involving $\overline{\epsilon}$, $\epsilon$ and low and high coefficients $a_{\low}$, $a_{\high}$ of factorials are bounded by a constant $C_r$ in the same way as in \eqref{epsilon_est}--\eqref{calc_ah}, and the low and high coefficients of powers of $\tau$ are defined by 
\begin{equation}
\begin{split}
b_{\low } &\coloneqq \frac{\tau^{i+j-r}}{\tau^{((n+l-r)\vee 0)/2}\tau^{((n+l+2-r)\vee 0)/2} \tau^{(i+j-n-l-r)\vee 0} },\\
b_{\high } &\coloneqq \frac{\tau^{i+j-r}}{\tau^{(n+l-r-1)\vee 0}\tau^{((i+j-n-l+1-r)\vee 0)/2} \tau^{((i+j-n-l+3-r)\vee 0)/2} } 
.
\end{split}
   \label{EQ74}
     \end{equation}
As in \eqref{bl}--\eqref{bh}, one can verify that 
\begin{equation}
b_{\low}  \lec \begin{cases} 
\tau \qquad &\text{ if } i+j-n-l\geq r, \\
1 &\text{ if } i+j-n-l \leq r-1,
\end{cases} \qquad 
b_{\high } \lec \begin{cases} 
\tau \qquad &\text{ if }n+l\geq r+1, \\
1 &\text{ if } n+l\leq r
,
   \label{EQ75}\end{cases}\hspace{0.8cm}
     \end{equation}
and thus we obtain
\eqnb\label{claim_A}
\begin{split}
\sum_{i+j \geq r } \frac{(i+j)^r}{(i+j)!}\tau^{i+j-r } \overline{\epsilon}^i \epsilon^{j} A &\lec \| v \|_{\widetilde{Y}(\tau )} \| u \|_{\widetilde{X}(\tau )} +\| v \|_{\widetilde{X}(\tau )}\| u \|_{\widetilde{Y}(\tau )},
\end{split}
\eqne
by the same calculation as \eqref{calc001} upon replacing $|\alpha |$, $|\beta |$ by $i+j$ and $n+l$, respectively. 

For the commutator term $B$, we use Lemma~\ref{L_com_curved} to obtain
\begin{equation}
\begin{split}
\| \p^{i-p} b_j  \p^{p+1} (v_k w)  \| &\leq \| \p^{i-p} b_j  \|_\infty  \sum_{n=0}^{p+1} {p+1 \choose n}  \| \p^n v_k \|_\infty  \| \p^{p+1-n } w   \| \\
&\lec (i-p+j)! K_1^{i-p} K_2^{j}  \sum_{n=0}^{p+1} {p+1 \choose n}  \| \p^n v_k \|_\infty  \|\p^{p+1-n } w   \|  
.
\end{split}
   \label{EQ76}
     \end{equation}
Noting that $\| \p^n v_k \|_\infty \lec \| \p^n v_k \|^{1/2} \| \p^{n+2} v_k \|^{1/2} +\| \p^n v_k \|  $ and 
\begin{equation}
\| \p^k v \|\lec F_{k,0} (v) \frac{k!}{k^r \overline{\epsilon }^n \tau^{(k-r)\vee 0} },
   \label{EQ77}
     \end{equation}
we deduce
\eqnb\label{curved_01}
\begin{split}
\sum_{i+j \geq r } c_{i,j} B&\lec \sum_{i+j \geq r } \sum_{p=0}^i  \sum_{n=0}^{p+1}  a b F_{n,0}(v)^{1/2} F_{n+2,0} (v)^{1/2} F_{p+1-n,0} (w) , 
\end{split}
\eqne
where
\begin{equation}\begin{split}
a&\coloneqq \frac{(i+j)^r}{(i+j)!}\overline{\epsilon}^i \epsilon^{j} {i\choose p} (i-p+j)! K_1^{i-p} K_2^j{p+1 \choose n }  \left( \frac{n!}{(n+1)^r \overline{\epsilon }^n} \right)^{\frac12} \left( \frac{(n+2)!}{(n+2)^r \overline{\epsilon }^{n+2}} \right)^{\frac12}  \frac{(p+1-n)!}{(p+2-n)^r \overline{\epsilon }^{p+1-n}} \\
& = \frac{(i+j)^r}{(i+j)!} \frac{i!(p+1)(i-p+j)!}{(i-p)!} \frac{((n+1)(n+2))^{1/2}}{(p+2-n)^r ((n+1)(n+2))^{r/2}} (     \overline{\epsilon} K_1)^{i-p} \overline{\epsilon}^{-2} (\epsilon K_2)^{j}
\end{split}
   \label{EQ78}
     \end{equation}
and 
\begin{equation}
b\coloneqq \frac{\tau^{i+j-r}}{\tau^{\frac12(n-r) \vee 0}\tau^{\frac12(n+2-r) \vee 0} \tau^{(p+1-n-r) \vee 0}} \lec 1.
   \label{EQ79}
     \end{equation}
The last inequality can be verified directly. In fact, it admits much more leeway than \eqref{bl}--\eqref{bh}; for example, the worst possible case, namely when $n\geq r$ and $p\geq n-r$, gives
\begin{equation}
b= \tau^{i+j-r-(n-r) -1-(p+1-n-r)} = \tau^{i-p+j +r-2} \leq 1. 
   \label{EQ80}
     \end{equation}
Note also that, for brevity, we have adjusted the definition of $a$ above by replacing $n^r$ by $(n+1)^r$ and $(p+1-n)^r$ by~$(p+2-n)^r$. This is due to the definition of $F_{i,j}(u)$ in \eqref{def_F} above, which has no prefactor for derivatives of order less than~$r$. For example, if $n=0$ then, without the adjustment, the right-hand side of \eqref{EQ78} would involve division by~$0$. By making the adjustment, we have used the inequality $n+1 \leq 2n$ for $n\geq 1$ (and hence ``$\lec$'' in \eqref{curved_01}).

We now handle~\eqref{curved_01}. Recall that, 
 in order to bound it by a product of analytic norms $\| v \|_{X(\tau )} \| w \|_{X(\tau )}$, the terms with $F$ need to be summed over $p$ and~$n$. We thus need to sum in $i$ and $j$ first. To this end, we note that
\begin{equation}
\sum_{i+j\geq r } \sum_{p=0}^i \leq \sum_{p\geq 0 } \sum_{i\geq p} \sum_{j\geq 0},
   \label{EQ81}
     \end{equation}
since the terms are nonnegative,
and so 
\begin{equation}
\begin{split}
\sum_{i\geq p} \sum_{j\geq 0} a &= (p+1) \frac{((n+1)(n+2))^{1/2}}{(p+2-n)^r ((n+1)(n+2))^{r/2}} \sum_{i\geq p} i^r (\overline{\epsilon} K_1)^i   \sum_{j\geq 0} \underbrace{\frac{(i-p+1) \ldots (i-p+j)}{(i+1)\ldots (i+j)}}_{\leq 1} j^r (\epsilon K_2)^j \lec 1,
\end{split}
   \label{EQ82}
     \end{equation}
provided $\overline{\epsilon}$ and $\epsilon$ are sufficiently small. Applying this in \eqref{curved_01} gives
\begin{equation}\begin{split}
\sum_{i+j \geq r } c_{i,j} B&\lec \sum_{p\geq 0 }  \sum_{n=0}^{p+1}   F_{n,0}(v)^{1/2} F_{n+2,0} (v)^{1/2} F_{p+1-n,0} (w) \lec \| v\|_{\widetilde{X}(\tau )}  \| w\|_{\widetilde{X}(\tau )}  
,
\end{split}
   \label{EQ83}
     \end{equation}
upon changing the order of summation and applying the Cauchy-Schwarz inequality to the $n$-sum. This and \eqref{claim_A} then prove the lemma.

\begin{remark}[Commutators are ok]\label{rem_comms}
{\rm
From the above proof, we note that the commutator terms, which arise whenever a derivative must be brought to the left of the tangential operator, are somewhat troublesome, as they require care with changing the order of summation. However, they are acceptable, as they are of lower order, which make all the coefficients summable with some more leeway. As a result, we do not need to be careful to obtain $\tau \| \cdot \|_{Y}+ \| \cdot \|_{H^r}$, but instead $\| \cdot \|_{\widetilde{X}}$, which is a smaller quantity. Thus, whenever we obtain the $\widetilde{X}$ norm using a straightforward change of the order of summation in a commutator term (or a sum of commutator terms), we  shall skip precise calculations.
}
\end{remark}

Similarly to Corollary~\ref{cor_products}, we can extend the product estimate \eqref{product_curved} to the case of products  where no derivative is lost, which in particular implies that the right-hand involves only the $\widetilde{X}$ norms (rather than one $\widetilde{Y}$ norm and one $\widetilde{X}$ norm).
Namely, we have
\eqnb\label{010_curved}
\sum_{\substack{ i+j\geq r \\ j\geq 2}} \frac{|\alpha |^r}{|\alpha |!} \tau^{|\alpha |-r}
\overline{\epsilon}^i \epsilon^{j} \| \p^i T^{j-2}  \nabla (\na f_1 \na f_2)  \|_{L^2} \lec_r \| f_1 \|_{\widetilde{X}(\tau )} \| f_2 \|_{\widetilde{X}(\tau )} ,
\eqne
\eqnb\label{011_curved}
\sum_{i+j\geq r } \frac{(i+j)^r}{(i+j)!} \tau^{i+j-r} \overline{\epsilon}^i \epsilon^{j} \| T^{i+j -1} (\na f_1 \na f_2)  \|_{L^2} \lec_r \| f_1 \|_{\widetilde{X}(\tau )} \| f_2 \|_{\widetilde{X}(\tau )} ,
\eqne
\eqnb\label{012_curved}
\sum_{i+j\geq r } \frac{(i+j)^r}{(i+j)!} \tau^{i+j-r} \overline{\epsilon}^i \epsilon^{j} \| T^{i+j-1} ( g_1 \cdot \nabla g_2)  \|_{H^1} \lec_r \| g_1 \|_{\tX (\tau )} \| g_2 \|_{\oY (\tau )}.
\eqne
Moreover,
\eqnb\label{013_acurved}
 \|  u  v \|_{X(\tau )} \lec_r \| u \|_{\tX (\tau )}  \| v \|_{{\tX}(\tau )}  ,
\eqne
\eqnb\label{013_curved}
 \|  u \cdot \nabla v \|_{X(\tau )} \lec_r \| u \|_{\tX (\tau )}  \| v \|_{{\oY}(\tau )}   .
\eqne
Note that there is no extra $\tau$ on the right-hand side above (as in \eqref{product_curved}, which involves the $\tY$ norm), as we need to take into account the case where the derivatives in $\p^\alpha (u\cdot \na v)$ all fall on~$\na v$.
\subsection{Pressure estimate}\label{sec_pressure_gen}

Here we prove the pressure estimate \eqref{pres_gen}, which is a consequence of the following statement.

\cole
\begin{lemma}[The pressure estimate in an analytic norm]\label{L_pressure_curved}
Consider the solution $p$ of the Neumann problem
  \begin{align}
  \begin{split}
   &\Delta p = \nabla f_1 : \nabla f_2
   \inon{in $\Omega$}
   \\&
   {\p _n} p = ((g_1\cdot \nabla g_2)\cdot \mathsf{n})|_{\partial\Omega}
   \inon{on $\partial\Omega$}
   ,
  \end{split}
   \label{EQ25_curved}
  \end{align}
with the normalizing condition $\int_\Omega p= 0$,
where $g_1,g_2$ are defined in $\Omega$,
with the compatibility condition
  \begin{equation*}
   \int_{\Omega}
     (\nabla f_1 : \nabla f_2)
   =
   \int_{\partial\Omega}
   ((g_1\cdot \nabla g_2)\cdot \mathsf{n})
   .
  \end{equation*}
 Then we have
  \begin{align}
  \begin{split}
   &
       \Vert \nabla p\Vert_{ X(\tau )}
   \lec
    \Vert f_1\Vert_{\tilde X(\tau)}
    \Vert f_2\Vert_{\tilde X(\tau)}
    +  \Vert g_1 \Vert_{\tX(\tau)}\Vert g_2 \Vert_{\oY(\tau)}
 ,
  \end{split}
   \label{EQ27_curved}
  \end{align}
provided ${\epsilon}$ and $\overline{\epsilon }/\epsilon$ are sufficiently small positive numbers. Moreover, if also $g_1\cdot \mathsf{n}= g_2 \cdot \mathsf{n}=0$, then 
 \begin{align}
  \begin{split}
   &
       \Vert \nabla p\Vert_{ X(\tau )}
   \lec
    \Vert f_1\Vert_{\tilde X(\tau)}
    \Vert f_2\Vert_{\tilde X(\tau)}
    +  \Vert g_1 \Vert_{\tX(\tau)}\Vert g_2 \Vert_{\tX(\tau)}
   .
  \end{split}
   \label{EQ27_acurved}
  \end{align}\colb
\end{lemma}
\colb

Note that the pressure estimate \eqref{pres_gen} follows from Lemma~\ref{L_pressure_curved} since in the case of quadratic term, namely $g_1=g_2=v$, we can use \eqref{EQ27_acurved} to obtain two $\tX$ norms, rather than an $\tX$ norm and a $\oY$ norm.

\begin{proof}[Proof of Lemma~\ref{L_pressure_curved}.]
We denote the left-hand side of \eqref{EQ27_curved} by
 \begin{equation}
 I \coloneqq \sum_{i+j\geq r } c_{i,j} \| \p^i T^j \p_k p \|   
 ,
   \label{EQ84}
     \end{equation}
where $k\in \{ 1,\ldots , d \}$ is fixed. We first assume that $i\geq 2$, and recall from \eqref{i2jany} that
\begin{equation}
\begin{split}
\| \p^i T^j \na p \|
   &\leq \| \p^i T^j \na p \|_{H^2}
   \\&
   \lec \| \p^{i-2} T^j \na (\na f_1 : \na f_2 ) \| + \| \p^{i-2} [\Delta ,T^j ] \na p \| + \| T\p^{i-2} T^j \na p \|_{H^1} + \| \p^{i-2} T^j \na p\|\\
&=: P+Q+R+S 
.
\end{split}
   \label{EQ85}
     \end{equation}
For $P$, we only consider the term when the derivative hits $f_1$, namely $\| \p^{i-2} T^j (D^2 f_1  \na f_2)  \| $, as the other one is analogous. 
We have 
\begin{equation}
\begin{split}
\| \p^{i-2} T^j D^2 f_1  \na f_2  \| &\lec \|  D^2 f_1 \cdot \p^{i-2} T^{j} \na f_2 \| +  \sum_{n=0}^{i-2} \sum_{\substack{l=0 \\ l+n \ne 0}}^j {i-2 \choose n}{j \choose l} \| \p^n T^l D^2 f_1 \cdot \p^{i-2-n} T^{j-l} \na f_2 \|\\
&=: P_1 + P_2
.
\end{split}
   \label{EQ86}
     \end{equation}
For $P_1$, we use the Cauchy-Schwarz inequality and the  interpolation $\| h \|_{4} \lec \| \nabla  h \|^{1/2} \| h \|^{1/2} + \| h \|$ to obtain
\begin{equation}
\begin{split}
\sum_{i+j\geq r } c_{i,j} P_1 &\lec \sum_{i+j\geq r } c_{i,j} \| D^3 f_1 \|^{1/2} \| D^2 f_1 \|^{1/2} \| \p^{i} T^j f_2 \|^{1/2} \| \p^{i-1} T^j  f_2 \|^{1/2} + (\text{Comm.})\\
&\lec  \| f_1 \|_{H^r}  \sum_{i+j\geq r } c_{i,j} \| \p^{i} T^j f_2 \|^{1/2}\| \p^{i-1} T^j  f_2 \|^{1/2}  + (\text{Comm.})\\
&\lec \| f_1 \|_{H^r} \| f_2 \|_{\widetilde{X}} ,
\end{split}
   \label{EQ87}
     \end{equation}
where, in the last line, we have applied the Cauchy-Schwarz inequality and the fact $\tau^{i+j-r} \leq \tau^{(i-1+j-r)\vee 0}$, which allowed us to obtain  $\| f_2 \|_{\widetilde{X}}$, since $(i+j)^r/(i+j)!\lec (i-1+j)^r/(i-1+j)!$, and the implicit constant may depend on $\overline{\epsilon},\epsilon$. Note that we have also neglected the lower order term in the interpolation inequality (which is simpler) and, as mentioned in Remark~\ref{rem_comms}, we have omitted a precise calculation of the estimate of the commutator term.

As for $P_2$, we use the interpolation $\| f_1 \|_\infty \lec \| D^2 f_1 \|^{1/2} \| f_1 \|^{1/2} + \| f_1 \|$ to obtain
\begin{equation}
\begin{split}
\| \p^n T^l D^2 f_1 \cdot \p^{i-2-n} T^{j-l} \na f_2 \| &\lec \| \p^n T^l D^2 f_1 \| \, \|  \p^{i-2-n} T^{j-l} \na f_2 \|_\infty \\
 &\lec \| \p^n T^l D^2 f_1 \| \, \| D^2 \p^{i-2-n} T^{j-l} \na f_2 \|^{1/2} \| \p^{i-2-n} T^{j-l} \na f_2 \|^{1/2}  \\
 &\lec \| \p^{n+2} T^l  f_1 \| \, \|  \p^{i+1-n} T^{j-l}  f_2 \|^{1/2} \| \p^{i-1-n} T^{j-l}  f_2 \|^{1/2}  + (\text{Comm.}) 
\end{split}
   \label{EQ88}
     \end{equation}
and thus 
\eqnb\label{201}
\begin{split}
\sum_{i+j\geq r } c_{i,j} P_2&\lec \sum_{i+j\geq r }   \sum_{n=0}^{i-2} \sum_{\substack{l=0 \\ l+n \ne 0}}^j c_{i,j} {i-2 \choose n}{j \choose l}\| \p^{n+2} T^l  f_1 \| \, \|  \p^{i+1-n} T^{j-l}  f_2 \|^{1/2} \| \p^{i-1-n} T^{j-l}  f_2 \|^{1/2}  + (\text{Comm.}) \\
&\lec \sum_{i+j\geq r }   \sum_{n=0}^{i-2} \sum_{\substack{l=0 \\ l+n \ne 0}}^j c_{i,j} {i-2 \choose n}{j \choose l}\| \p^{n+2} T^l  f_1 \| \, \|  \p^{i+1-n} T^{j-l}  f_2 \|^{1/2} \| \p^{i-1-n} T^{j-l}  f_2 \|^{1/2}  + (\text{Comm.}) \\
&\lec \sum_{i+j\geq r }   \sum_{n=0}^{i-2} \sum_{\substack{l=0 \\ l+n \ne 0}}^j  F_{n+2,l} (f_1) F_{i+1-n,j-l}(f_2)^{1/2} F_{i-1-n,j-l} (f_2)^{1/2}.
\end{split}
\eqne
In the last step, we used 
\begin{equation}
\tau^{i+j -r } \leq \tau^{(n+2+l-r)\vee 0} \left( \tau^{(i-n+1+j-l)\vee 0} \tau^{(i-1-n+j-l)\vee 0} \right)^{\frac12},
   \label{EQ89}
     \end{equation}
which can be verified directly (here we also used $(n,l)\ne 0$). Note that, in order to obtain the terms with $F$ above, we also estimated the factorials as 
\begin{equation}
\begin{split}
c_{i,j} &{i-2 \choose n}{j \choose l} \cdot \frac{(n+2+l)!}{(n+2+l)^r}\left( \frac{(i-1-n+j-l)! (i-1-n+j-l)!}{(i-1-n+j-l)^r (i-1-n+j-l)^r} \right)^{\frac12} \\
&= \underbrace{\frac{{n+2+l \choose n+2}{i-2-n+j-l \choose i-2-n }}{{i+j \choose i}}}_{\leq 1} \cdot \underbrace{ \frac{(n+1)(n+2)}{i(i-1)}}_{\leq 1}\cdot \frac{1}{(i-1-n+j-l)(i-n+j-l)^{1/2} (i+1 -n+j-l)^{1/2}} \\
&\hspace{7cm}\cdot\underbrace{\frac{(i+j)^r}{(n+2+l)^r (i+1-n+j-l)^{\frac{r}2}(i-1-n+j-l)^{\frac{r}{2}}} }_{\lec 1 }\\
&\lec 1
,
\end{split}
   \label{EQ90}
     \end{equation}
due to Vandermonde's identity, and by $(a+b)^r\leq (2ab)^r$ for $a,b\geq 1$. Note that we have also omitted a detailed analysis of the commutator terms in \eqref{201}, which would require further rearranging. Applying the Cauchy-Schwarz inequality to \eqref{201}, we obtain
\begin{equation}
\sum_{i+j\geq r} c_{i,j} P_2 \lec \| f_1 \|_{\widetilde{X}} \| f_2 \|_{\widetilde{X}} ,
   \label{EQ91}
     \end{equation}
as required.

As for $Q$, we use the commutator decomposition \eqref{202} to obtain
\begin{equation}
\begin{split}
\sum_{i+j\geq r} c_{i,j} Q&\lec  \sum_{i+j\geq r}  \sum_{l=0}^{j-1} \sum_{|\alpha |=j-l} \sum_{m=0}^{i-2} \frac{(i+j)^r }{(i+j)!} \tau^{i+j-r} \overline{\epsilon}^i \epsilon^j  \frac{j!}{\alpha! l!} \| \p^{i-2-m} (b_{\alpha_1} b_{\alpha_2} ) \|_\infty \frac{(i-2)!}{m! (i-2-m)!} \| \p^{m+2} T^l \nabla p\|, 
\end{split}
   \label{EQ92}
     \end{equation}
where from 
\eqnb\label{203}
b_{\alpha_1} \cdot \nabla (b_{\alpha_2} \cdot \nabla )= \sum_{i,j}b_{\alpha_1,i} b_{\alpha_2,j} \p_i \p_j  + ((b_{\alpha_1} \cdot \nabla ) b_{\alpha_2} )\cdot \nabla
\eqne
we have considered only the first term.
We note that 
\begin{equation}
\begin{split}
\frac{\| \p^{i-2-m} (b_{\alpha_1} b_{\alpha_2} ) \|_\infty }{\alpha! (i-2-m)!} &\leq \frac{1}{\alpha! (i-2-m)!}\sum_{n=0}^{i-2-m}  {i-2-m \choose n} \| \p^n b_{\alpha_1} \|_\infty \| \p^{i-2-m-n} b_{\alpha_2} \|_\infty \\
&\leq \frac{1}{\alpha! } \sum_{n=0}^{i-2-m}  \frac{(n+\alpha_1)! (i-1-m-n+\alpha_2)!}{n! (i-1-m-n)!} K_1^{i-1-m} K_2^{|\alpha |}   \\
&\leq 2^{i-1-m+|\alpha |} (i-m)  K_1^{i-1-m} K_2^{|\alpha |}   \\
&\leq (4K_1)^{i-1-m} (4K_2)^{|\alpha |},
\end{split}
   \label{EQ93}
     \end{equation}
which gives 
\eqnb\label{labelQ}
\begin{split}
\sum_{i+j\geq r} c_{i,j} Q&\lec  \sum_{i+j\geq r}  \sum_{l=0}^{j-1} \underbrace{\sum_{|\alpha |=j-l}}_{\lec 2^{j-l}} \sum_{m=0}^{i-2} \frac{(i+j)^r }{(i+j)!} \tau^{i+j-r} \overline{\epsilon}^i \epsilon^j  \frac{j!}{ l!}  \frac{(i-2)!}{m! } (4K_1)^{i-1-m} (4K_2)^{j-l}\| \p^{m+2} T^l \nabla p\|\\
&\lec \sum_{m,l\geq 0} F_{m+2,l}(\na p)  \sum_{\substack{i\geq m+2 \\ j\geq l+1 \\ i+j\geq r}} \underbrace{\frac{(m+2+l)! (i+j)^r j! (i-2)!}{(m+2+l)^r (i+j)! l! m! }}_{=:q} \overline{\epsilon}^{i-2-m} \epsilon^{j-l} (4K_1)^{i-2-m} (8K_2)^{j-l} \\
&\lec \epsilon \| \na p \|_{\widetilde{X}}
;
\end{split}
\eqne
in the last line we have used
\begin{equation}
q = \frac{ {m+2+l \choose m+2}}{{i+j \choose i}} \frac{(i+j)^r}{(m+2+l)^r } \underbrace{\frac{(m+1)(m+2)}{i(i-1)}}_{\leq 1} \lec (i+j-(m+2+l))^r \lec_r (4 K_1)^{i-2-m} (8K_2)^{j-l},
   \label{EQ94}
     \end{equation}
by the binomial inequality 
\eqnb\label{binom_ineq}
\begin{split}
{m+2+l \choose m+2} &= \frac{(m+2+l)!}{(m+2)! l!}  =  \frac{(m+2+j)!}{(m+2)! j!} \underbrace{\frac{l+1}{m+2+l+1} \ldots \frac{j}{m+2+j}}_{\leq 1}  \\
&\leq \frac{(i+j)!}{i! j!}  \underbrace{\frac{m+3}{m+3+j}\ldots  \frac{i}{i+j}}_{\leq 1} \leq {i+j \choose i}
,
\end{split}
\eqne
for $i\geq m+2$, $j\geq l$, as well as the elementary inequality\footnote{This is obvious for $a\in [b+1,2b)$ and otherwise, if $a\in [kb,(k+1)b)$ for some $k\geq 2$, we have $a/b\leq k+1 \leq 2(k-1) \leq 2(k-1)b  \leq 2(a-b) $, as required.}$\frac{a}{b} \leq 2(a-b)$ for $a\geq b+1 \geq 2$. The second part of $Q$, arising from the second term in \eqref{203}, is estimated analogously.\\

As for $R$, we note that the case $i=2$ is trivial, since then
\begin{equation}
R\lec \| \p T^{j+1} \na p\| + \| T^{j+1} \na p\|
,
   \label{EQ95}
     \end{equation}
and so
\eqnb\label{labelR0}
\sum_{2+j\geq r} c_{2,j} R \leq \frac{\overline{\epsilon}}{\epsilon } \| \na p \|_{\widetilde{X}} .
\eqne
For $i\geq 3$, we write
\begin{equation}
 \nabla T \p^{i-2} T^j \na p = \p^{i-1} T^{j+1} \na p - \sum_{n=0}^{i-3} {i-2 \choose n} \na \left(  \p^{i-2-n} b  \,\p^{n+1} T^j \na p \right)
,
   \label{EQ96}
     \end{equation}
and thus
\begin{equation}\begin{split}
R&\lec \| T \p^{i-2} T^j \na p \| + \| \p^{i-1} T^{j+1} \na p \| + \sum_{n=0}^{i-3} {i-2 \choose n} \left( \| \p^{i-2-n} b\|_\infty \| \p^{n+2} T^j \na p\|+ \| \p^{i-1-n} b \|_\infty \| \p^{n+1}T^j \nabla p \| \right)\\
&=: R_1 + R_2 + R_3 + R_4
.
\end{split}
   \label{EQ97}
     \end{equation}
Clearly,
\eqnb\label{labelR1}
\sum_{\substack{i+j\geq r\\ i\geq 3}} c_{i,j} (R_1+R_2) \leq \left( \overline{\epsilon} + \frac{\overline{\epsilon}}{\epsilon} \right)\| \na p\|_{\widetilde{X}}.
\eqne
Moreover,
\eqnb\label{labelR2}
\begin{split}
\sum_{\substack{i+j\geq r\\ i\geq 3}} c_{i,j} R_3 &\lec   \sum_{\substack{i+j\geq r\\ i\geq 3}}\sum_{n=0}^{i-3}  \frac{(i+j)^r (n+j)!}{(n+j)^r (i+j)! } \frac{(i-2)!}{n! } (\overline{\epsilon} K_1)^{i-2-n}  F_{n+2,j} (\na p)\\
&\lec \sum_{n,j\geq 0 }\frac{j^r }{(n+j)^r} F_{n+2,j} (\na p) \sum_{i\geq n+3} \frac{{n+j \choose n}}{{i+j \choose i} (i-1)i } i^r (\overline{\epsilon} K_1)^{i-2-n}  \\
&\lec \overline{\epsilon} K_1 \sum_{n,j\geq 0 } F_{n+2,j} (\na p)
\\&
\lec \overline{\epsilon} \| \na p\|_{\widetilde{X}},
\end{split}
\eqne
where we used \eqref{binom_ineq} in the last inequality. A similar calculation shows that $\sum_{{i+j\geq r, i\geq 3}} c_{i,j} R_4 \lec \overline{\epsilon} \| \na p\|_{\widetilde{X}} $. \\

As for $S$, we have that
\eqnb\label{labelS}
\sum_{i+j\geq r} c_{i,j} \| \p^{i-2}T^j \na p\| \leq \overline{\epsilon}^2 \sum_{i+j\geq r} F_{i-2,j}(\na p) \leq \overline{\epsilon}^2 \| \na p \|_{\widetilde{X}},
\eqne
since $\frac{(i+j)^r}{(i+j)!} \leq \frac{(i-2+j)^r}{(i-2+j)!} $.\\

It remains to consider the case $i\in \{ 0,1\}$. To this end, we note that
\eqnb\label{i1_case}
\sum_{j\geq r-i} c_{i,j} \| \p^i T^j \na p \| \lec\epsilon \| \nabla p \|_{\widetilde{X}} + \sum_{j\geq r-i}c_{i,j}  \| T^{j+i-1} p \|_{H^2}.
\eqne
Indeed, for $i=1$, we use \eqref{bounds_b} to get
\eqnb\label{204}
\| \p [ T^j , \nabla ] p \| = \| \p (b_j \cdot \nabla )p \| \lec j! K_2^j  \| D^2 p \| + (j+1)! K_2^j \| \nabla p \|
\eqne
and so
\begin{equation}
\sum_{j\geq r-1} c_{1,j} \| \p [ T^j , \nabla ]  p \| \lec \overline{\epsilon }\epsilon^2 \| \nabla p \|_{H^1}  \lec \epsilon \| \nabla p \|_{\widetilde{X}}
.
   \label{EQ98}
     \end{equation}
Similarly, in the case $i=0$, we have $\sum_{j\geq r} c_{0,j} \| [ T^j , \nabla ]  p \| \lec \epsilon \| \nabla p \|_{\widetilde{X}}$, from which \eqref{i1_case} follows. 

In order to control $\| T^{j+i-1} p \|_{H^2}$, we recall that $p$ satisfies $\Delta p = \nabla f_1 : \nabla f_2$ with the Neumann boundary condition $\p_n p  =(g_1\cdot \nabla g_2)\cdot \mathsf{n}$. We focus on the case $i=1$ first. Applying $T^j$, we see that $T^j p$ satisfies
\begin{equation}
\begin{cases}
\Delta T^j p =T^j (\na f_1  \na f_2 ) + [\Delta , T^j ] p\quad &\text{ in }\Omega ,\\
\p_{\mathsf{n}} T^j p = T^j \left(  (g_1\cdot \na g_2)\cdot \mathsf{n} \right) -\sum_{k=1}^j {j\choose k} T^{j-k} \na p \cdot T^k \mathsf{n} + [\na ,T^j ] p \cdot \mathsf{n} &\text{ on } \p \Omega 
.
\end{cases}
   \label{EQ99}
     \end{equation}
Applying the $H^2$ elliptic estimate thus gives
\eqnb\label{205}
\begin{split}
\| T^j p \|_{H^2} &\lec \| T^j (\na f_1  \na f_2)\| + \| [\Delta , T^j ] p \| + \| T^j ((g_1 \cdot \na) g_2 \cdot \mathsf{n}  ) \|_{H^{1/2}(\p \Omega)}
\\&\indeq
+ \sum_{k=1}^j {j \choose k} \| T^{j-k} \na p \cdot T^k \mathsf{n} \|_{H^1} + \| [\nabla , T^j ]p \cdot \mathsf{n} \|_{H^1}\\
&=: M_1 + M_2 + M_3 + M_4 + M_5.
\end{split}
\eqne
In order to bound $M_1 $--$M_5$, we note that the only terms in \eqref{205} where we do not gain a derivative (from a commutator or $\mathsf{n}$) are those containing $v,w,g$, and so we expect all the remaining terms to be bounded by~$\epsilon \| \na p \|_{\widetilde{X}}$.

We have $\sum_{j\geq r-1} c_{1,j} M_1 \leq \| f_1 \|_{\widetilde{X}} \| f_2 \|_{\widetilde{X}}$, by the product estimate \eqref{011_curved}, and, similarly to term $Q$ above we have, using \eqref{202},
\begin{equation}
\begin{split}
\sum_{j\geq r-1} c_{1,j} M_2 &\lec \sum_{j\geq r-1} \frac{(1+j)^r}{(1+j)!} \tau^{j+1-r} \overline{\epsilon} \epsilon^{j} \sum_{l=0}^{j-1} \sum_{|\alpha |=j-l}\frac{j!}{\alpha! l!} \underbrace{ \| b_{\alpha_1} b_{\alpha_2} \|_\infty}_{\lec \alpha! K^{j-l}} \| \p^2 T^l p \|\\
&\lec  \sum_{j\geq r-1}\sum_{l=0}^{j-1} \frac{(1+j)^r}{(1+j)!} \tau^{j+1-r} \overline{\epsilon} \epsilon^{j}  \sum_{|\alpha |=j-l}\frac{j!}{\alpha! l!} \| b_{\alpha_1} b_{\alpha_2} \|_\infty \| \p^2 T^l p \|\\
&\lec  \sum_{l\geq r-2} F_{1,l} (\na p) \underbrace{ \sum_{j\geq l+1} (\epsilon K) ^{j-l}(j-l)^{r+1}  }_{\lec \epsilon } + \sum_{j\geq r-1}\sum_{l=0}^{j-1} \frac{(1+j)^r}{(1+j)!} \tau^{j+1-r} \overline{\epsilon} \epsilon^{j}  \sum_{|\alpha |=j-l}\frac{j!}{l!} K^{j-l} \| \p [ T^l,\na ] p \|\\
&\lec \epsilon \| \na p\|_{\widetilde{X}} +\overline{\epsilon} \sum_{j\geq r-1}\sum_{l=0}^{j-1} j^{r+1} (\epsilon K)^{j} \| \na p \|_{H^1}\\
&\lec \epsilon \| \na p\|_{\widetilde{X}},
\end{split}
   \label{EQ100}
     \end{equation}
where $K\coloneqq \max \{ K_1, K_2\}$, and we noted, as in \eqref{204}, that
\begin{equation}
\| \p [T^l, \na ]p \| \leq \| \na b_l \|_\infty \| \na p\| + \| b_l \|_{\infty } \| D^2 p \|\lec (1+l)! K^l \| \na p \|_{H^1}.
   \label{EQ101}
     \end{equation}
For $M_3$, we use the trace estimate on $\Omega'$, the product estimates  \eqref{013_acurved}--\eqref{013_curved}, and the analyticity of $\mathsf{n}$ (recall~\eqref{n_is_ana})
 to obtain
  \begin{align}
  \begin{split}
\sum_{j\geq r-1} c_{1,j} M_3 &\lec  \sum_{j\geq r-1} c_{1,j} \left(\|\p T^j ((g_1 \cdot \na )g_2\cdot \mathsf{n} ) \|_{L^2 (\Omega')} + \| T^j ((g_1 \cdot \na )g_2\cdot \mathsf{n} ) \|_{L^2 (\Omega')}  \right)
\\&
\lec \| (g_1 \cdot \nabla )g_2 \cdot \mathsf{n} \|_{\tX (\tau;\Omega')} \lec \| g_1 \|_{\tX } \| \na g_2 \|_{\tX} \lec  \| g_1 \|_{\tX } \| g_2 \|_{\oY}  ,
  \end{split}
   \label{M3_calc}
  \end{align}
  \colb 
while for $M_4$ we have
\begin{equation}\begin{split}
 \sum_{j\geq r-1} c_{1,j}& \sum_{k=1}^{j} {j\choose k} \| \na ( T^{j-k} \na p \cdot T^k \mathsf{n} ) \|\\ &\lec \sum_{k\geq 1} \sum_{\substack{j\geq k \\j\geq r-1} } \frac{(1+j)^{r-1}}{k! (j-k)!} \overline{\epsilon} \epsilon^j \tau^{(1+j-k-r)\vee 0} \left( \| \p T^{j-k } \na p\| \| T^k \mathsf{n} \|_{\infty} +\|  T^{j-k } \na p\| \| \p T^k \mathsf{n} \|_{\infty} \right)  \\
 &\lec \epsilon \| \na p \|_{\widetilde{X}} \sum_{k\geq 1}  \epsilon^k\left(  \frac{k^r}{k!} \| T^k \mathsf{n} \|_\infty + \overline{\epsilon }\frac{(1+k)^r}{(1+k)!} \|\p  T^k \mathsf{n} \|_\infty \right)\\
 &\lec_\mathsf{n} \epsilon \| \na p \|_{\widetilde{X}}, 
 \end{split} 
   \label{EQ102}
     \end{equation}
where we used the fact \eqref{n_is_ana} that $\mathsf{n}$ is an analytic function with a  lower bound on the analyticity radius. A similar and easier argument gives the same bound when there is no $\na$ on the left-hand side, which yields $\sum_{j\geq r-1} c_{1,j} M_4 \lec \epsilon \| \na p \|_{\widetilde{X}}$.

As for $M_5$, we get 
 \begin{equation}
 \| \na ([\na , T^j ]p \cdot \mathsf{n})\| \lec  \| \na (b_j \cdot \na ) p \| + \| (b_j \cdot \na ) p \| \| \mathsf{n} \|_{H^r} \lec (j+1)! K^j \| \na p \|_{H^1}, 
   \label{EQ103}
     \end{equation}
 and similarly for $\| [\na , T^j ]p \cdot \mathsf{n} \|$, which leads to
 \eqnb\label{labelM5}
 \sum_{j\geq r-1} c_{1,j} M_5 \lec \| \na p \|_{H^1} \sum_{j\geq r-1} (j+1)^r \tau^{j+1-r} \overline{\epsilon } \epsilon^j K^j \lec \epsilon \| \na p \|_{\widetilde{X}}.
 \eqne

We now focus on $\| T^{j+i-1} p \|_{H^2} $ in the case $i=0$. Recall from~\eqref{i1_case} that we then need to control $\| T^{j-1} p \|_{H^2}$, and the arguments remain analogous. Indeed, note that in the above arguments, concerned with the case $i=1$, we have not utilized the fact that $\overline{\epsilon }/\epsilon \ll 1$. We only needed this earlier for the derivative reduction procedures in the case $i\geq 2$ above (in \eqref{labelR1}, for example), where one does not gain a derivative. There we only used the fact that $\epsilon $ is sufficiently small for the absorption of the resulting terms $\epsilon \| \na p \|_{\widetilde{X}}$. In the control of $\| T^{j-1} p \|_{H^2}$, similarly to the case $i=1$ above, only the combined powers of $\overline{\epsilon}$ and $\epsilon$ are relevant, showing that the computation in the case $i=0$ is analogous.\\ \colb

Finally, in order to justify the last claim in the lemma, we take advantage of the derivative reduction (see Lemma~\ref{L001} below) to write $(g_1\cdot \nabla) g_2 \cdot \mathsf{n} = - g_{1,l} g_{2,m} \p_{lm} \phi$, where $\phi$ is the signed distance function. We replace \eqref{M3_calc} by 
\begin{equation}\begin{split}
\sum_{j\geq r-1} c_{1,j} M_3 &\lec  \sum_{j\geq r-1} c_{1,j} \left(\|\p T^j (g_1g_2 D^2\phi) \|_{L^2 (\Omega')} + \| T^j (g_1g_2 D^2\phi) \|_{L^2(\Omega')}  \right) \\
&\lec \|g_1g_2 D^2\phi \|_{\tX (\tau;\Omega')}\lec \| g_1  g_2 \|_{\tX}\lec \| g_1 \|_{\tX } \| g_2 \|_{\tX}   
\end{split}
   \label{EQ104}
     \end{equation}
in the case $i=1$, where we also used \eqref{D2phi_is_ana}, and we make a similar adjustment in the case $i=0$.
\end{proof}

Analogously to Corollary~\ref{cor_extra_p_flat}, we now extend the above pressure estimate to the case of a simpler boundary condition.

\cole
\begin{corollary}\label{cor_extra_p_gen}
Suppose that $p$ satisfies
\eqnb
\begin{split}
\Delta p &=0 \quad \text{ in }\Omega,\\
\p_{\mathsf{n}} p &=f \cdot \mathsf{n} \quad \text{ on }\p \Omega.
\end{split}
\eqne
Then $\| \na p \|_{\tX} \lec \| f \|_{\tX}$.
\end{corollary}
\colb

\begin{proof}[Proof of Corollary~\ref{cor_extra_p_gen}]
We note that, similarly to \eqref{labelQ}--\eqref{i1_case}, we have
\eqnb\label{tempcor}
\| \na p\|_{\tX } \lec \left( {\epsilon}+ \frac{\overline{\epsilon}}{\epsilon } \right) \| \na p \|_{\tX} +\sum_{j\geq r}c_{0,j}  \| T^{j-1} p \|_{H^2}+ \sum_{j\geq r-1}c_{1,j}  \| T^{j} p \|_{H^2} .
\eqne
As for the last term we obtain, similarly to \eqref{205}--\eqref{labelM5} (the only difference is in \eqref{M3_calc}, where ``$g_1\cdot \nabla g_2$'' is replaced by ``$f $'') we obtain
\begin{equation}
\sum_{j\geq r-1}c_{1,j}  \| T^{j} p \|_{H^2} \lec \epsilon \| \na p \|_{\tX } + \| f \cdot \mathsf{n} \|_{\tX (\tau ; \Omega' )}\lec \epsilon \| \na p \|_{\tX } + \| f \|_{\tX },
   \label{EQ105}
     \end{equation}
where we used \eqref{013_acurved} and the analyticity assumption~\eqref{n_is_ana} on~$\Omega$. A similar argument for the penultimate term in \eqref{tempcor} gives the claim.
\end{proof}\colb

\subsection{Derivative reduction lemma}
Here we prove a reduction lemma, which is due to Temam~\cite{T}.

\cole
\begin{lemma}[Derivative reduction]
\label{L001}
Let $f$ and $g$ be sufficiently smooth vector fields
in $\Omega$ such that
  \begin{equation}
   f\cdot\mathsf{n}=g\cdot\mathsf{n}=0
   \inon{on $\partial\Omega$}
   .
   \label{EQ02a}
  \end{equation}
Then we have
  \begin{equation}
   (f\cdot \nabla g)\cdot \mathsf{n}
   = - f_i g_j \partial_{ij}\phi
   \inon{on $\partial\Omega$}
   ,
   \llabel{EQ01}
  \end{equation}
where $\phi$ is the signed distance function
taking positive values outside and negative values inside
(so that $\mathsf{n}=\nabla \phi$).
\end{lemma}
\colb

Note that $|\nabla \phi|=1$ in a neighborhood of~$\partial\Omega$.
For simplicity, we denote
  \begin{equation}
   \phiij = \partial_{ij} \phi
   \comma i,j=1,2
.
   \label{EQ65}
  \end{equation}

\begin{proof}[Proof of Lemma~\ref{L001}]
The boundary $\partial\Omega$ is a solution of the equation
$\phi(x)=0$, and since also
  \begin{equation}
   g \cdot \nabla \phi = 0
   \inon{on $\partial\Omega$}   
   ,
   \llabel{EQ03}
  \end{equation}
by \eqref{EQ02a} for $g$, noting also that $\nu=\nabla \phi$, we have
  \begin{equation}
   \nabla (g\cdot \nabla \phi) = k \nabla \phi
   \inon{on $\partial\Omega$}
   ,
   \llabel{EQ04}
  \end{equation}
i.e.,
  \begin{equation}
   \partial_i(g_j\partial_{j}\phi)= k \partial_{i} \phi   
   \inon{on $\partial\Omega$}
   ,
   \llabel{EQ05}
  \end{equation}
where $k$ is a scalar valued function in a neighborhood
of~$\partial\Omega$.
Therefore, we obtain
  \begin{align}
  \begin{split}
   (f\cdot \nabla g)\cdot \mathsf{n}
   &=    (f\cdot \nabla g)\cdot \nabla \phi
   = f_i \partial_{i} g_j \partial_{j} \phi
   = f_i \partial_{i} (g_j \partial_{j} \phi)
     - f_i g_j \partial_{ij} \phi
   = f_i k \partial_{i}\phi
     - f_i g_j \partial_{ij} \phi      
   = - f_i g_j \partial_{ij} \phi      
  \end{split}
  \llabel{EQ06}
  \end{align}
on $\partial\Omega$,
where we used \eqref{EQ02a} for $f$ in the last step.
\end{proof}\colb

\section{Global well-posedness in $2$D}\label{sec_global}
Here we prove Theorem~\ref{T02}. For simplicity, we take $\tau_0=1$. 
Suppose that
\begin{equation}
\| \omega_0 \|_\infty + \| v_0 \|_{\widetilde{X}(1)} \leq C_0,
   \label{EQ106}
     \end{equation}
where $\omega_0 \coloneqq \mathrm{curl}\,v_0$, and 
\begin{equation}
\| \ou (t) \|_{\widetilde{X}(1)} + \| \ou (t) \|_{\overline{Y}(1)} \leq f(t) \qquad \text{ for all }t\geq 0, 
   \label{EQ107}
     \end{equation}
where $f\colon [0,\infty ) \to [0,\infty )$ is such that
\eqnb\label{f_assum_1}
\| f\|_{L^2} + \| f \|_\infty \leq C_1
\eqne
for some $C_0,C_1\geq1$.
We note that the a~priori estimate \eqref{002} gives 
\begin{equation}
\frac{\d }{\d t } \| v \|_{\widetilde{X}(\tau )} - \dot \tau \| v \|_{Y(\tau )} \leq \widetilde{C} \bigl( \| v \|_{\widetilde{X}(\tau )} \| v \|_{\widetilde{Y}(\tau )} +f \bigl(  \| v \|_{\widetilde{X}(\tau )} +  \| v \|_{\overline{Y}(\tau )} \bigr) + f^2 \bigr)
   \label{EQ03}
     \end{equation}
for some $\widetilde{C}\geq1$.
Thus, if
\eqnb\label{tau_eq}
-\dot \tau \geq 2\widetilde{C} \left( \| v \|_{\widetilde{X}}\tau +f \right) 
\eqne
(recall \eqref{temp23}), then
\eqnb\label{023}
\frac{\d }{\d t } \| v \|_{\widetilde{X}(\tau )} - \frac{\dot \tau}2 \| v \|_{Y(\tau )} \leq \widetilde{C} \bigl( \| v \|_{\widetilde{X}(\tau )} \| v \|_{H^3} +f \bigl(  \| v \|_{\widetilde{X}(\tau )} +  \| v \|_{H^3} \bigr) + f^2 \bigr)
.
\eqne
This suggests that the $\widetilde{X}$ norm (as well as the time integral of the $Y$ norm) is under control as long as the $H^3$ norm is.
As for the $H^3$ norm, we first note that 
\begin{equation}
\frac{\d }{\d t} \| \omega + \oo \|_{\infty } \leq \| \ou \|_\infty \|\nabla  \omega + \nabla\oo \|_{\infty }
   \label{EQ108}
     \end{equation}
(since $\omega + \oo$ satisfies  the transport equation with the velocity $v+\ou$, and $v\cdot \mathsf{n}=0$ on the boundary), where $\omega \coloneqq \mathrm{curl}\, v$ and $\oo \coloneqq \mathrm{curl}\,\ou$.
Therefore,
\eqnb\label{aaa}
\| \omega \|_\infty \leq \| \omega_0 + \oo (0) \|_\infty +  \| \oo \|_\infty + \int_0^t \| \ou \|_\infty \| \nabla \omega +\nabla \oo \|_{\infty }  \leq C
\eqne
for all $t>0$ such that 
\eqnb\label{onemore}
\int_0^t f (s) \| v (s) \|_{W^{2,\infty}} \leq 1,
\eqne
where $C\geq 1$ is a constant depending on both $C_0$ and~$C_1$.

Here we have assumed for a moment that $v$ exists and is regular for all
times. The inequality \eqref{aaa} and the well-known elliptic bound $\| \nabla v \|_\infty \lec 1+ \| \mathrm{curl}\, v \|_\infty (1+\log^+ \| v \|_{H^3})$ (see (6) in \cite{F}) allows us to derive the $H^3$~estimate of \eqref{intro_euler_v},
\begin{equation}
\frac{\d }{\d t} \| v \|_{H^3} \lec \| v \|_{H^3} \left( \| \nabla v \|_\infty + f \right) +f\| v \|_{{Y}} + f^2 \lec \| v \|_{H^3} \left( 1+ \log^+ \| v \|_{H^3} \right) +f\| v \|_Y + 1,
   \label{EQ109}
     \end{equation}
where the implicit constant also depends on $C_0, C_1$, and the ``$f\| v \|_Y$'' term comes from the ``$\ou \cdot \nabla v$'' term in \eqref{intro_euler_v}, as  $\|v\|_{H^4}\lec \| v\|_Y+\|v\|_{H^3}$. We note that the appearance of the $H^4$ norm above is a consequence of the derivative loss described in \eqref{der_loss} above. We thus obtain $C_2\geq1$ (depending on $C_0,C_1$) such that 
\begin{equation}
\| u (t)\|_{H^3} \leq C_2 \exp \exp (C_2 t)
   \label{EQ110}
     \end{equation}
for all $t\geq 0$ such that
\eqnb\label{f_assum_3cons}
\int_0^t f(s) \| v (s)\|_Y \d s \leq 1.
\eqne
Consequently, \eqref{023} gives 
  \begin{align}
  \begin{split}
   &
\sup_{s\in [0,t]}\| v(s) \|_{\widetilde{X}(\tau )} - \int_0^t \dot \tau  \| v  (s)\|_{Y(\tau )} \d s
\\&\indeq
\lec \exp \left( \exp \exp (C_2 t) +\int_0^t f\right) \left( 1+ \| v_0 \|_{\widetilde{X}} +\| v_0 \|_{H^3} \int_0^t f \exp \exp  (C_2 s) \d s \right) 
  ,
  \end{split}
   \label{023a} 
  \end{align}
and so there exists $C_3\geq 1$ (depending on $C_0,C_1$) such that
\eqnb\label{023b}
\sup_{s\in [0,t]}\| v(s) \|_{\widetilde{X}(\tau )}  -\int_0^t \dot \tau  \| v  (s)\|_{Y(\tau )} \d s  \leq C_3 \exp \exp \exp (C_3 t) 
\eqne
for all $t\geq 0$ such that \eqref{tau_eq} and
\eqnb\label{f_assum_2}
\int_0^t f(s) \exp \exp (C_2 s ) \d s \leq 1
\eqne
hold.
We emphasize that $\tau $ (and so also $\| \cdot \|_{\widetilde{X}}$ and $ \| \cdot \|_{Y}$) is a function of time. Now that we have estimated the expected growth of $\| v \|_{\widetilde{X}}$, we can define~$\tau$. Namely, in light of \eqref{tau_eq}, we see that there exists $C_4\geq1$ such that 
\eqnb\label{tau_def}
\tau (t) \coloneqq  \exp \left( - C_4 \int_0^t \exp \exp \exp (C_4 s )\d s \right)  
\eqne
satisfies $\tau_0=1$, as assumed, and \eqref{tau_eq}, provided \eqref{023b} and that $f$ satisfies
\eqnb\label{f_assum_2a}
C_5 f(t) \exp^4 (C_5 t) \leq 1
    \comma t\geq 0,
\eqne
where $C_5 \geq C_4$ is a large constant. 
This choice of $\tau$ first lets us determine another constraint for $f$ guaranteeing~\eqref{onemore}.
Namely, noting that there exists $C_6>C_5$ such that 
\begin{equation}
\| v(s) \|_{W^{2,\infty}} \lec \tau(s)^{-10}\| v (s) \|_{\tX (\tau (s))} \lec C_6 \exp^4 (C_6 s)
   \label{EQ111}
     \end{equation}
for all $t\geq 0$ for which \eqref{023b} holds, we see that \eqref{onemore} follows if 
\eqnb\label{onemore1}
C_6 \int_0^t f(s) \exp^4 (C_6 s)  \d s \leq 1.
\eqne
Second, the choice \eqref{tau_def} determines the integrability of $\| v\|_{Y}$ in time, namely we have 
\begin{equation}
-\dot \tau (t) \geq  C_7^{-1}  \exp \left(- \exp \exp \exp (C_7 t) \right) 
   \label{EQ112}
     \end{equation}
for some $C_7\geq C_6$, and so \eqref{023} implies that
\eqnb\label{023c}
\int_0^t \frac{\| v \|_{Y(\tau )}}{C_7  \exp^4 (C_7 s)} \d s \lec \int_0^t C_3  \exp^3 (C_3 s )\,\d s
,
\eqne
from where $\int_0^t \frac{\| v \|_{Y(\tau )}}{C_8 \exp^4 (C_8 s)} \d s \leq 1$  if  $C_8>C_7$ is sufficiently large.

This shows that, choosing $\tau$ as in \eqref{tau_def}, we obtain the global-in-time a~priori bound \eqref{023b}, provided $\ou$ is such that $f$ satisfies 
\begin{equation}
f(t) \, C_8 \exp^4 (C_8 t) \leq 1 \qquad \text{ for all }t>0 ,
   \label{EQ113}
     \end{equation}
so that $f$ satisfies \eqref{f_assum_1}, \eqref{f_assum_3cons}, \eqref{f_assum_2}, \eqref{f_assum_2a}, and~\eqref{onemore1}. This way, all the above estimates hold unconditionally for all $t\geq 0$ (by applying a simple continuity argument), and the a~priori bound \eqref{023b}, which is valid for all $t\geq 0$, lets  us modify the local well-posedness proof of Section~\ref{sec_flat_conclude} to obtain the global well-posedness result with 
\begin{equation}
\| v \|_{L^\infty ((0,\infty ); \widetilde{X}(\tau ))} + \| -\dot \tau v \|_{L^1 ((0,\infty ); {Y} (\tau ))} <\infty,
   \label{EQ114}
   \end{equation}
with $\tau $ chosen in \eqref{tau_def}, as required.
\section*{Data availability statement}
The paper has no associated data.

\section*{Acknowledgments}
The authors are grateful to James Kelliher for interesting discussions and constructive comments. This work is based upon work supported by the National Science Foundation under Grant No.~DMS-1928930 while the authors were in residence at the Simons Laufer Mathematical Sciences Institute (formerly MSRI) in Berkeley, California, during the summer of~2023.
IK was supported in part by the NSF grant DMS-2205493,
while
MS was supported in part by the grant PRIN-PNRR~P202254HT8.

\colb

\colb

\ifnum\sketches=1
\newpage
\begin{center}
  \bf   Notes?\rm
 \begin{align}
  \begin{split}
   &
   \Vert u\Vert_{X(\tau)}
   \coloneqq 
  \sum_{i+j\geq r } c_{ij} 
    \Vert \p^i T^j  u\Vert , \qquad \text{ where } c_{ij} \coloneqq \frac{(i+j)^r}{(i+j)!}\tau^{i+j-r }
  \overline{\epsilon}^i \epsilon^j,
    \\&
   \Vert u\Vert_{\tX(\tau)}
   \coloneqq 
   \Vert u\Vert_{X(\tau)}
   + \Vert u\Vert_{H^{r}}
  \\&
   \Vert u\Vert_{Y(\tau)}
   \coloneqq 
   \sum_{ i+j\geq r+1}
    \frac{i+j}{
     \tau
    }c_{ij}
    \Vert \p^i T^j  u\Vert
    \\&
   \Vert u\Vert_{\widetilde{Y}(\tau)}
   \coloneqq 
  \tau  \Vert u\Vert_{Y(\tau)}
   + \Vert u\Vert_{H^{r}} \\&
   \Vert u\Vert_{\oY (\tau)}
   \coloneqq 
  \Vert u\Vert_{Y(\tau)}
   + \Vert u\Vert_{H^{r}},
  \end{split}
   \eqref{EQ19c}
  \end{align}
 \begin{align}
  \begin{split}
   &
   \Vert u\Vert_{X(\tau)}
   \coloneqq 
   \sum_{ |\alpha |\geq r}
    \frac{
     |\alpha |^{r}
        }{
     |\alpha |!
    }
    \tau^{|\alpha |-r}
   \epsilon^{\alpha_2}
    \Vert \p^\alpha u\Vert
    ,
    \\&
   \Vert u\Vert_{\tX(\tau)}
   \coloneqq 
   \Vert u\Vert_{X(\tau)}
   + \Vert u\Vert_{H^{r}}
  ,
  \\&
   \Vert u\Vert_{Y(\tau)}
   \coloneqq 
   \sum_{ |\alpha |\geq r+1}
    \frac{
     |\alpha |^{r+1}
        }{
     |\alpha |!
    }
    \tau^{|\alpha |-r-1}
    \epsilon^{\alpha_2}
    \Vert\p^\alpha u\Vert
    ,
    \\&
   \Vert u\Vert_{\widetilde{ Y}(\tau)}
   \coloneqq 
  \tau  \Vert u\Vert_{Y(\tau)}
   + \Vert u\Vert_{H^{r}}
   ,
    \\&
   \Vert u\Vert_{\overline{Y}(\tau)}
   \coloneqq 
  \Vert u\Vert_{Y(\tau)}
   + \Vert u\Vert_{H^{r}}.
  \end{split}
   \eqref{EQ19}
  \end{align}

\end{center}
\huge
\colb

\fi


\end{document}